\documentclass[11pt]{article}
\usepackage{pb-diagram}
\usepackage{amsmath,amsthm,amsfonts,amssymb,latexsym}
\usepackage{amsmath}

\usepackage{a4wide}

\newtheorem{thm}{Theorem}[section]

\newtheorem{lem}[thm]{Lemma}
\newtheorem{prop}[thm]{Proposition}

\newtheorem{rmk}[thm]{Remark}

\numberwithin{equation}{section}

\newcommand{\Ker}{\mathrm{Ker}}

\begin{document}


\title{On the Riemannian geometry of tangent Poisson-Lie group}

\author{Foued Aloui\\
E.S.S.T.H.Sousse.\\ University
of Sousse\\
Email:\\
foued$_-$aloui@yahoo.fr
 \\
 }
 \date{}

 \date{}
\maketitle
\begin{abstract}
Let  $(G,\Pi_{G},\tilde{g})$ be a Poisson-Lie group equipped with a left invariant pseudo-Riemannian metric $\tilde{g}$ and let $(TG,\Pi_{TG},\tilde{g}^{c})$ be the Sanchez de Alvarez  tangent Poisson-Lie group of $G$ equipped with the left invariant pseudo-Riemannian metric $\tilde{g}^{c}$, complete lift of $\tilde{g}$. In this paper, we express  respectively  the Levi-Civita connection, curvature and metacurvature of  $(TG,\Pi_{TG},\tilde{g}^{c})$  in terms of the Levi-Civita connection, curvature and metacurvature  of the basis   Poisson-Lie group $(G,\Pi_{G},\tilde{g})$ and we prove that the space of differential forms $\Omega^{*}(G)$ on $G$ is a differential graded Poisson algebra if, and only if,  $\Omega^{*}(TG)$  is a differential graded Poisson algebra . Moreover, we prove that the triplet $(G,\Pi_{G},\tilde{g})$ is a pseudo-Riemannian Poisson-Lie group if, and only if, $(TG,\Pi_{TG},\tilde{g}^{c})$  is also a pseudo-Riemannian Poisson-Lie group and we give an example of 6-dimensional pseudo-Riemannian Sanchez de Avarez tangent Poisson-Lie group.
 Lie group and Lie algebra. \\
 A.M.S. Subject classification : 53D17. 58B20. 70G65.\\
\end{abstract}
\bigskip
\thispagestyle{empty}
\section{Introduction}
Riemannian geometry of tangent bundles of smooth manifolds is an important area in mathematics and physics which has begun by Sasaki \cite{sa}. He used a Riemannian manifold $(M,\tilde{g})$ and constructed a Riemannian metric on $TM$ defined by the horizontal and vertical lifts.\\
Another way of prolonging the tensor fields (Riemannian and pseudo-Riemannian metrics)
and affine connections on the tangent bundle of a manifold M is the use of complete
and vertical lifts.
In \cite{yako}, Yano and Kobayachi used the complete and vertical lifts to construct a pseudo-Riemannian metric on the tangent bundle $TM$ by a pseudo-Riemannian metric on the base manifold $M.$ In fact, if 
$(M,\tilde{g})$ is a pseudo-Riemannian metric then the complete lift of $\tilde{g}$ is a pseudo-Riemannian metric on $TM$ denoted by $\tilde{g}^{c}$, given for any vector field $X,Y \in \chi(M)$ by:
\begin{equation*} 
\begin{array}{rcl}
 \tilde{g}^{c}(X^{v}, Y^{v})  &=& 0,\\
\tilde{g}^{c}(X^{c}, Y^{v}) &=& (\tilde{g}(X, Y))^{v} \\
\tilde{g}^{c}(X^{c}, Y^{c}) &=& (\tilde{g}(X, Y))^{c},
\end{array}
\end{equation*}
where $X^{v}$ is the vertical lift of $X$ and $X^{c}$ is the complete lift of $X$ (for more details see \cite{ML-PR}\cite{Yano}).\\~~~~\\
Moreover, Poisson manifolds play a fundamental role in Hamiltonian dynamics, where they serve as a phase space. A Poisson bracket on a manifold $M$ is a Lie bracket $\{ , \}_{M}$ on the space of differentiable functions $\mathcal{C}^{\infty}(M)$ satisfying the Leibniz identity,
\begin{equation*}
\{f,gh\}_{M} = \{f,g\}_{M}h + g\{f,h\}_{M}.
\end{equation*}
From the Leibniz identity there exists a bivector field $\Pi_{M} \in \wedge^{2}TM$, called Poisson tensor such that
\begin{equation} \label{t}
\{f,g\}_{M} = \Pi_{M}(df,dg).
\end{equation}
A Poisson manifold is a manifold equipped with a Poisson bracket.\\~~~~~~\\
An important class of Poisson  manifolds equipped with pseudo-Riemannian metrics is the family of Poisson-Lie groups equipped with left invariant pseudo-Riemannian metrics.\\
The notion of Poisson-Lie group  was first introduced by Drinfel'd \cite{D}\cite{Dr} and Semenov-Tian- Shansky \cite{Seme}.  Semenov,  Kosmann-Schwarzbach and Magri \cite{ko} used Poisson-Lie
groups to understand the Hamiltonian structure of the group of dressing transformations
of certain integrable systems. These Poisson-Lie groups play the role of
symmetry groups.\\~~~~~~~\\
 In \cite{M.N}, M.Boumaiza and N.Zaalani showed that  if $(G,\Pi_{G})$ is a Poisson-Lie group then the tangent bundle $(TG,\Pi_{TG})$ of $G$ with its tangent Poisson structure $\Pi_{TG}$ defined in the sense of Sanchez de Alvarez \cite{sa de al} is a Poisson Lie group. This Poisson-Lie group $(TG,\Pi_{TG})$ is called Sanchez de Alvarez tangent Poisson-Lie group of $G$ \cite{al-za}.\\~~~~~\\
In  \cite{al-za} the author and N.Zaalani study the Riemannian geometry and the compatibility of the Sanchez de Alvarez tangent Poisson-Lie group  $(TG,\Pi_{TG})$ of $G$ equipped with the natural left invariant Riemannian metric (or the left invariant Sasaki metric, or the left invariant Cheeger-Gromoll metric) in terms of a Poisson-Lie group equipped with a left invariant Riemannian metric. The non-compatibility  between the Sanchez de Alvarez Poisson structure  and the natural left invariant Riemannian  metric (except in the trivial case $\Pi_{G} = 0$) on $TG$ \cite{al-za} this pushed us to look for another Riemannian metric on the tangent bundle $TG$ for which there is  a compatibility with the Sanchez de Alvarez Poisson structure. For this reason we are interested in this paper  by the complete Riemannian metric on $TG$. Then, we consider the left invariant pseudo-Riemannian metric $\tilde{g}^{c}$ on $TG$, complete lift of that of $G$ and we study the geometry of the triple $(TG,\Pi_{TG},\tilde{g}^{c})$ and its relations with the geometry of $(G,\Pi_{G},\tilde{g})$.\\~~~~~~\\
This paper is organized as follows: In section 2, we recall basic definitions and facts about contravariant connections, curvatures, metacurvatures, generalized Poisson brackets, pseudo-Riemannian Poisson-Lie group. In section 3, we express respectively  the Levi-Civita connection, curvature, metacurvature and the generalized Poisson brackets of $(TG,\Pi_{TG},\tilde{g}^{c})$ in terms of the Levi-Civita connection, curvature, metacurvature and the generalized Poisson brackets of $(G,\Pi_{G},\tilde{g})$ respectively and in section 4, we show that $(G,\Pi_{G},\tilde{g})$ is a pseudo-Riemannian Poisson-Lie group if, and only if, $(TG,\Pi_{TM},\tilde{g}^{c})$ is a pseudo-Riemannian Poisson-Lie group. In section 5, we give an example of 6-dimensional pseudo-Riemannian Sanchez de Avarez tangent Poisson-Lie group.
\section{Preliminaries}
\subsection{Contravariant connections and curvatures}
 Contravariant connections on Poisson manifolds were defined by Vaisman \cite{Va} and studied in details by Fernandes \cite{Fe}. This notions appears extensively in the context of noncommutative deformations \cite{Ha,Haw}.\\
Let $(M,\Pi_{M})$ be a Poisson manifold. We associate with the Poisson tensor $\Pi_{M}$ the anchor map $\Pi^{\sharp}_{M} :  T^{*}M \rightarrow TM$  defined by $\beta(\Pi^{\sharp}_{M}(\alpha)) = \Pi_{M}(\alpha,\beta) $ and the  Koszul bracket $[ , ]_{M}$ on the space of differential 1-forms $\Omega^{1}(M)$ given by:
\begin{equation*} \label{koz}
\begin{array}{rcl}
[\alpha , \beta]_{M} &=& \mathcal{L}_{\Pi^{\sharp}_{M}(\alpha)}\beta - \mathcal{L}_{\Pi^{\sharp}_{M}(\beta)}\alpha - d(\Pi_{M}(\alpha,\beta)).
\end{array}
\end{equation*}
A contravariant connection on $M,$ with respect to $\Pi_{M}$, is a $\mathbb{R}$-bilinear map
\begin{eqnarray*}
\mathcal{D}^{M} : \Omega^{1}(M))   \times  \Omega^{1}(M) & \to & \Omega^{1}(M) \nonumber \\
(\alpha,\beta) & \mapsto & \mathcal{D}^{M}_{\alpha}\beta
\end{eqnarray*}
such that for all $f \in \mathcal{C}^{\infty}(M)$\\
\begin{equation*}
\mathcal{D}^{M}_{f\alpha}\beta = f \mathcal{D}^{M}_{\alpha}\beta \hspace{2mm} \hbox{and} \hspace{2mm}
 \mathcal{D}^{M}_{\alpha}(f\beta) = f \mathcal{D}^{M}_{\alpha}\beta + \Pi^{\sharp}_{M}(\alpha)(f)\beta.
 \end{equation*}
 The torsion $\mathcal{T}^{M}$ and the curvature $\mathcal{R}^{M}$ of a contravariant connection $\mathcal{D}^{M}$  are formally
identical to the usual ones:
\begin{equation*}
\mathcal{T}^{M}(\alpha,\beta) = \mathcal{D}^{M}_{\alpha}\beta - \mathcal{D}^{M}_{\beta}\alpha - [\alpha,\beta]_{M}
\end{equation*}
\begin{equation} \label{curva}
 \mathcal{R}^{M}(\alpha,\beta)\gamma = \mathcal{D}^{M}_{\alpha}\mathcal{D}^{M}_{\beta}\gamma -
\mathcal{D}^{M}_{\beta}\mathcal{D}^{M}_{\alpha}\gamma -
\mathcal{D}^{M}_{[\alpha,\beta]_{M}}\gamma.
\end{equation}
These are respectively (2,1) and (3,1)-type tensor fields. When $\mathcal{T}^{M} \equiv 0$ (resp.
$\mathcal{R}^{M} \equiv 0$), $\mathcal{D}^{M}$ is called torsion-free (resp. flat).\\
Now let $(M,\tilde{g})$ be a pseudo-Riemannian manifold. The metric $\tilde{g}$ defines the musical isomorphisms
\begin{eqnarray*}
\sharp : TM & \to & T^{*}M \nonumber \\
X & \mapsto & \tilde{g}(X,.)
\end{eqnarray*}
 and its inverse $\sharp^{-1}$. We define the contravariant pseudo-Riemannian metric  $g$ associated to $\tilde{g}$ defined by:
$$g(\alpha, \beta) = \tilde{g}(\sharp^{-1}(\alpha), \sharp^{-1}(\beta)).$$
For each contravariant Riemannian metric $g$, there exists a unique contravariant connection $\mathcal{D}^{M}$ associated with
$(\Pi_{M},g)$ such that $\mathcal{D}^{M}$ is torsion-free, i.e.,
\begin{equation*} \label{torsion}
[\alpha,\beta]_{M} = \mathcal{D}^{M}_{\alpha}\beta -  \mathcal{D}^{M}_{\beta}\alpha,
\end{equation*}
and the metric $g$ is parallel with respect to $\mathcal{D}^{M}$, i.e.,
\begin{equation*} \label{comp}
\Pi^{\sharp}_{M}(\alpha) g(\beta, \gamma) =  g(\mathcal{D}^{M}_{\alpha}\beta, \gamma) + g(\beta, \mathcal{D}^{M}_{\alpha}\gamma).
\end{equation*}
The connection $\mathcal{D}^{M}$ is called the Levi-Civita  contravariant connection  associated with  $(\Pi_{M},g)$. It has appeared first in \cite{Bou} and  defined by the Koszul formula:
\begin{equation} \label{Ko}
\begin{array}{rcl}
 2 g(\mathcal{D}^{M}_{\alpha}\beta, \gamma)  &=& \Pi^{\sharp}_{M} (\alpha)g(\beta, \gamma)  + \Pi^{\sharp}_{M} (\beta)g(\alpha, \gamma) - \Pi^{\sharp}_{M}(\gamma)g(\alpha, \beta) \\ &\quad \quad & + \quad 
 g([\alpha, \beta]_{M}, \gamma)  +  g([\gamma, \alpha]_{M}, \beta)  +  g([\gamma, \beta]_{M}, \alpha).
\end{array}
\end{equation}
We say that  $\mathcal{D}^{M}$ is locally symmetric if $\mathcal{D}^{M}\mathcal{R}^{M} = 0,$ i.e.,  if for any $\alpha,\beta,\gamma,\delta \in \Omega^{1}(M),$ we have:
\begin{equation} \label{local}
(\mathcal{D}^{M}_{\alpha}\mathcal{R}^{M})(\beta,\gamma)\delta := \mathcal{D}^{M}_{\alpha}(\mathcal{R}^{M}(\beta,\gamma)\delta) - \mathcal{R}^{M}(\mathcal{D}^{M}_{\alpha}\beta,\gamma)\delta - \mathcal{R}^{M}(\beta,\gamma)\mathcal{D}^{M}_{\alpha}\delta  - \mathcal{R}^{M}(\beta,\mathcal{D}^{M}_{\alpha}\gamma)\delta = 0.
\end{equation}
\subsection{Generalized Poisson bracket}
Let $(M,\Pi_{M})$ be a Poisson manifold and $\mathcal{D}^{M}$ a torsion-free and flat connection with respect to $\Pi_{M}$. In \cite{Haw}, E.Hawkins showed that such a connection
defines an $\mathbb{R}$-bilinear bracket on the space of differential forms
$\Omega^{*}(M)$ denoted also by  $\{ , \}_{M}$, such that:
\begin{enumerate}
  \item The bracket $\{ , \}_{M}$ is  antisymmetric, i.e.,
  \begin{equation*}
  \{ \sigma, \upsilon \}_{M} = -(-1)^{\deg(\sigma) \deg(\upsilon)} \{\upsilon , \sigma\}_{M}.
  \end{equation*}
  \item  $\{ , \}_{M}$, satisfies the product rule, i.e.,
  \begin{equation*}
  \{\sigma , \upsilon \wedge \nu\}_{M} = \{ \sigma, \upsilon \}_{M} \wedge \nu + (-1)^{\deg(\sigma) \deg(\upsilon)} \upsilon \wedge \{\sigma , \nu\}_{M}.
  \end{equation*}
  \item The exterior differential $d$ is a derivation with respect to $\{ , \}_{M}$, i.e.,
  \begin{equation*}
  d\{ \sigma,\upsilon \}_{M} = \{d\sigma ,\upsilon \}_{M} + (-1)^{\deg(\sigma)}\{ \sigma, d\upsilon\}_{M}.
  \end{equation*}
  \item For any $f, g \in \mathcal{C^{\infty}}(M)$ and for any $\sigma \in \Omega^{*}(M)$, the bracket $\{f,g\}_{M}$ coincides  with the initial Poisson bracket on  $M$ and
      \begin{center}
      $\{f , \sigma\}_{M} = \mathcal{D}^{M}_{df}\sigma$.
      \end{center}
      \end{enumerate}
      This bracket is given for any $\alpha,\beta \in \Omega^{1}(M)$ by \cite{Bah},
       \begin{equation} \label{bracket}
     \{\alpha,\beta\}_{M} = -\mathcal{D}^{M}_{\alpha}d\beta -\mathcal{D}^{M}_{\beta}d\alpha + d \mathcal{D}^{M}_{\beta}\alpha + [\alpha,d\beta]_{M},
    \end{equation}
    where $[ , ]_{M}$ is the generalized Koszul bracket on $\Omega^{*}(M)$ satisfying the Leibnuz identity, i.e.,
  \begin{equation} \label{id}
  [\sigma , \upsilon \wedge \nu]_{M} = [\sigma, \upsilon ]_{M} \wedge \nu + (-1)^{(\deg(\sigma) -1) \deg(\upsilon)} \upsilon \wedge [\sigma , \nu]_{M}.
  \end{equation}
   Note that the generalized Koszul bracket for the forms differential is the analogy of the Schouten Nijenhuis bracket for the multivectors fields.\\
       We call this bracket $\{ , \}_{M}$ a generalized pre-Poisson
bracket associated with the contravariant connection $\mathcal{D}^{M}.$
      E.Hawkins  showed that there
exists a (2,3) tensor $\mathcal{M}^{M}$ symmetric in the contravariant indices
and antisymmetric in the covariant indices such that the generalized pre-Poisson bracket satisfies the graded Jacobi identity,
\begin{equation*}
\{\sigma,\{\upsilon,\nu\}_{M}\}_{M} - \{\{\sigma,\upsilon\}_{M},\nu\}_{M} - (-1)^{\deg(\sigma)\deg(\upsilon)} \{\upsilon,\{\sigma,\nu\}_{M}\}_{M} = 0,
\end{equation*}
 if, and only if,
      $\mathcal{M}^{M}$ is identically zero.\\ $\mathcal{M}^{M}$ is called metacurvature of $\mathcal{D}^{M}$ and given by
      \begin{equation} \label{equa2.3}
      \mathcal{M}^{M}(df,\alpha,\beta) = \{ f, \{ \alpha, \beta\}_{M}\}_{M} - \{\{f ,\alpha \}_{M} , \beta\}_{M} - \{ \{f , \beta\}_{M}, \alpha\}_{M}.
      \end{equation}
      If $\mathcal{M}^{M}$ vanishes identically, the contravariant connection $\mathcal{D}^{M}$ is called metaflat and  the
bracket $\{ , \}_{M}$ is called the generalized Poisson bracket associated with $\mathcal{D}^{M}$, making $\Omega^{*}(M)$ a differential graded Poisson algebra (for more details see \cite{Haw}).
\subsection{Vertical and complete lifts of tensor fields to the tangent bundle}
In this subsection we recall the vertical and complete lifts of tensor fields in the sense of \cite{Yano}  from a manifold $M$ to its tangent bundle $TM$.\\
If $\alpha$ is a 1-form sur $M$, it is regarded, in a natural way, as a function sur $TM,$  which we denote by $\iota\alpha.$
       For any $f \in   \mathcal{C^{\infty}}(M)$, the vertical lift (resp. the complete lift) of  $f$ sur $M$ to the tangent bundle $TM$ is the function
         $f^{v}$ (resp. $f^{c}$)  defined by
        \begin{equation*}
        f^{v} = f \circ \pi \hspace{2mm} (\hbox{resp}. \hspace{2mm} f^{c} = \iota(df)),
        \end{equation*}
        where $\pi : TM \rightarrow M$ is the natural projection and $df$ the tangent map of $f$.\\
         For any vector field $X \in \chi(M)$, the vertical lift (resp. the complete lift) of $X$ sur $M$ to $TM$ is the  vector field $X^{v}$ (resp. $X^{c}$) defined by,
   \begin{equation} \label{vert}
    X^{v}(\iota\alpha) = (\alpha(X))^{v},\hspace{2mm} (\hbox{resp}.\hspace{1mm} X^{c}(f^{c}) = (X(f))^{c},
    \end{equation}
    where $\alpha \in \Omega^{1}(M)$ and $f \in   \mathcal{C^{\infty}}(M)$.\\
    For $(x,X) \in TM$, the Kernel of $d\pi : TTM \rightarrow TM$ at point $(x,X),$
    $$\Ker(d\pi(x,X)) = V_{(x,X)},$$
    is the vertical subspace of $T_{(x,X)}TM.$\\~~~\\
     For any $X \in \chi(M) $ and $f \in   \mathcal{C^{\infty}}(M)$ we have:
        \begin{equation} \label{eqs}
      (f + g)^{v} = f^{v} + g^{v}, \hspace{1mm} (f + g)^{c} = f^{c} + g^{c}, \hspace{1mm}  X^{v}f^{v} = 0,\hspace{1mm} X^{v}f^{c} = X^{c}f^{v} = (Xf)^{v} \hspace{1mm} \hbox{and} \hspace{2mm} X^{c}f^{c} = (Xf)^{c}.
        \end{equation}  
          For any 1-form $\alpha \in \Omega^{1}(M),$ we denote by $\alpha^{v} \in \Omega^{1}(TM)$ (resp.  $\alpha^{c} \in \Omega^{1}(TM)$) the vertical lift (resp. the complete lift) of $\alpha.$ For any 1-forms $\alpha,\beta \in \Omega^{1}(M)$ we have:
          \begin{equation}
          (\alpha + \beta)^{v} = \alpha^{v} + \beta^{v} \hspace{2mm} \hbox{and} \hspace{2mm} (\alpha + \beta)^{c} = \alpha^{c} + \beta^{c}.
          \end{equation}
         If we denote by $(x_{i})$ $(i=1,....,n)$ local coordinates on $M$ and if $\alpha = \alpha_{i}dx_{i},$ then $\alpha^{v} = \alpha_{i}dx_{i}$. Thus $\alpha^{v}$ is precisely the pull back of $\alpha$ to $TM,$ that is $$\alpha^{v} = \pi^{*}(\alpha),$$
         where $\pi^{*}: T^{*}M \rightarrow T^{*}TM$.\\
          For any differential forms $\mu$ and $\nu$ we have:
          \begin{equation} \label{id1}
          (\mu \wedge \nu)^{c} = \mu^{c} \wedge \nu^{v} + \mu^{v} \wedge \nu^{c}, \hspace{1mm} (\mu \wedge \nu)^{v} = \mu^{v} \wedge \nu^{v}, \hspace{1mm} d(\mu^{c}) = (d\mu)^{c}, \hspace{1mm} \hbox{and} \hspace{1mm} d(\mu^{v}) = (d\mu)^{v}.
          \end{equation}
\subsection{Pseudo-Riemannian Poisson-Lie group}
A Lie group $G$ is called a Poisson-Lie group if it is also a Poisson manifold such that the product $$ m :G \times G \rightarrow G : (x,y) \mapsto xy$$ is a Poisson map, where $G \times G$ is equipped with the product Poisson structure.\\
      Let $G$ be a Poisson Lie group with Lie algebra $(\mathfrak{g},[ , ]_{\mathfrak{g}})$ and $\Pi_{G}$ the Poisson tensor
on $G$. Pulling $\Pi_{G}$ back to the identity element $e$ of $G$ by the left translations, we get a map $\Pi_{G}^{l} : G \rightarrow \mathfrak{g} \wedge \mathfrak{g}$, defined by  $\Pi^{l}_G(x) =  (L_{x^{-1}})_{*}\Pi_G(x),$ where $(L_{x})_{*}$ denotes the tangent map of the left translation $L_{x}$ of $G$ by $x$. The intrinsic derivative $$\xi := d_{e}\Pi^{l}_{G} : \mathfrak{g} \rightarrow \mathfrak{g} \wedge \mathfrak{g}$$ of $\Pi^{l}_{G}$ at $e$ is a 1-cocycle relative to the adjoint representation
of $\mathfrak{g}$ on $\mathfrak{g} \wedge \mathfrak{g}$. The dual map of $\xi$ is a Lie bracket  $[ , ]_{\mathfrak{g}^{*}} : \mathfrak{g}^{*} \times \mathfrak{g}^{*} \rightarrow \mathfrak{g}^{*}$  on $\mathfrak{g}^{*}.$ It is well-known that $(\mathfrak{g},\mathfrak{g}^{*})$ is a Lie bialgebra.\\~~~~~\\
Let $(G,\Pi_{G})$ be a Poisson-Lie group with Lie bialgebra $(\mathfrak{g},\mathfrak{g}^{*})$. Let $a$ be a bilinear, symmetric and non-degenerate form on $\mathfrak{g}^{*}$ and $g$  the
contravariant pseudo-Riemannian given by $g = (L_{x})_{*}a$. We say that $(G,\Pi_{G},g)$  is a pseudo-Riemannian Poisson-Lie group if, and only if, the Poisson tensor $\Pi_{G}$ and the metric $g$ are compatible in the sense given by M.Boucetta in \cite{Bouc}, as follows: 
  \begin{equation} \label{Ri}
[ Ad_x^*(A_\alpha\gamma+ad^*_{\Pi^{l}_{G}(x)(\alpha)}\gamma),Ad_x^*(\beta)]_{\mathfrak{g}^*}+ [Ad_x^*(\alpha),Ad_x^*(A_\beta\gamma+ad^*_{\Pi^{l}_{G}(x)(\beta)}\gamma)]_{\mathfrak{g}^*}=0,
\end{equation}
for any $x \in G$ and for any $\alpha,\beta,\gamma \in \mathfrak{g}^{*}$, where  $A$ is the infinitesimal Levi-Civita connection associated with $([ , ]_{\mathfrak{g}^{*}} , a)$.\\ Note that  the infinitesimal Levi-Civita connection $A$ is the restriction of the Levi-Civita contravariant connection $\mathcal{D}^{G}$ to $\mathfrak{g}^{*} \times \mathfrak{g}^{*}$ and  is given for any $\alpha,\beta,\gamma \in \mathfrak{g}^{*},$  by:
\begin{equation} \label{e1}
 2 a(A_{\alpha}\beta, \gamma)  =
 a([\alpha, \beta]_{\mathfrak{g^{*}}}, \gamma) + a([\gamma,
  \alpha]_{\mathfrak{g^{*}}}, \beta)  +  a([\gamma, \beta]_{\mathfrak{g^{*}}}, \alpha).
\end{equation}
In \cite{Bouc}, M.Boucetta showed that if $(G,\Pi_{G},g)$ is a pseudo-Riemannian Poisson Lie group then,  its dual Lie algebra $(\mathfrak{g}^{*},[ , ]_{\mathfrak{g}^{*}},a)$ equipped with the form $a$ is a pseudo-Riemannian Lie algebra, i.e, for any $\alpha,\beta,\gamma \in \mathfrak{g}^{*}$ we have,
 \begin{equation} \label{Riemannian Lie}
[ A_{\alpha}\beta,\gamma]_{\mathfrak{g}^{*}}+ [\alpha,A_{\gamma}\beta]_{\mathfrak{g}^{*}}=0.
\end{equation} 
   \section{Riemannian geometry of Sanchez de Alvarez tangent Poisson-Lie  group.}
 Let $G$ be a n-dimensional Lie group with multiplication $m : G \times G \rightarrow G : (x,y) \mapsto xy$ and with Lie algebra $\mathfrak{g}$. We denotes by $L_{y} : G \rightarrow G : x \mapsto yx$ for the left translation and $R_{y} : G \rightarrow G : x \mapsto xy$ for the right translation of $G$ by $y$.\\
     The tangent map of $m,$
     \begin{equation} \label{loi}
    Tm : TG \times TG \mapsto TG : (X_x,Y_y) \mapsto T_yL_xY_y + T_xR_yX_x,
    \end{equation}
    defines a Lie group structure on $TG$ with identity element $(e,0)$ and with Lie algebra the semi-direct product of Lie algebra $\mathfrak{g} \rtimes \mathfrak{g}$, with  bracket \cite{M.N}:
    \begin{equation}
  [(X,Y),(X^{'},Y^{'})]_{\mathfrak{g} \rtimes \mathfrak{g}} = ([X,X^{'}]_{\mathfrak{g}},[X,Y^{'}]_{\mathfrak{g}} + [Y,X^{'}]_{\mathfrak{g}}), \hspace{2mm} \hbox{where} \hspace{2mm} (X,X^{'}),(Y,Y^{'}) \in \mathfrak{g} \rtimes \mathfrak{g}.
    \end{equation}
    Note that the Lie algebra $\mathfrak{g} \rtimes \mathfrak{g}$ of $TG$ is given by \cite{yakob}:
    \begin{equation*}
    \mathfrak{g} \rtimes \mathfrak{g}  = \{ (X,Y) = X^{c} + Y^{v} \mid X,Y \in \mathfrak{g} \}.
    \end{equation*}
    Let $(G,\Pi)$ be a Poisson-Lie group  with Lie bilagebra $(\mathfrak{g},\mathfrak{g}^{*})$ and $TG$ the tangent bundle of $G$.
   By M.Boumaiza and N.Zaalani \cite{M.N} the tangent bundle $TG$ of $G$  with the multiplication (\ref{loi}) and with its tangent Poisson structure $\Pi_{TG}$ defined in the sense of  Sanchez de Alvarez  \cite{sa de al}, is a Poisson-Lie group with Lie bialgebra $(\mathfrak{g} \rtimes \mathfrak{g}, \mathfrak{g}^{*} \ltimes \mathfrak{g}^{*})$, where
     $\mathfrak{g}^{*} \ltimes \mathfrak{g}^{*}$ is the semi-direct product Lie algebra  with bracket:
\begin{equation} \label{e2}
[(\alpha,\beta),(\alpha',\beta')]_{{\mathfrak{g}}^*\ltimes {\mathfrak{g}}^*}=
([\alpha,\beta']_{\mathfrak{g}^{*}}+[\beta,\alpha']_{\mathfrak{g}^{*}},[\beta,\beta']_{\mathfrak{g}^{*}}),\;\;
 \hbox{where} \hspace{2mm} (\alpha,\beta),\;(\alpha',\beta')\in {\mathfrak{g}}^*\times {\mathfrak{g}}^*.
\end{equation}
    It has been shown that the complete and vertical lifts of any left invariant vector fields of $G$ are left invariant
    fields on the Lie group $TG$ (see proposition 1.3
page 183 of \cite{Yano}). In fact if $(X_1,...,X_n)$ is a basis for the Lie algebra $\mathfrak{g}$ of $G$ then
    $\{X_1^v = (0,X_1),...,X_n^v = (0,X_n),X_1^c = (X_1,0),...,X_n^c = (X_n,0)\}$ is a basis for the Lie algebra $\mathfrak{g} \rtimes \mathfrak{g}$ of $TG.$ \\
    Let $\tilde{g}$ be a left invariant pseudo-Riemannian metric on $G$. Then the complete lift of $\tilde{g}$ is a left invariant pseudo-Riemannian metric $\tilde{g}^{c}$ on $TG$ given for any $(X,Y),(X^{'},Y^{'}) \in \mathfrak{g} \rtimes \mathfrak{g}$ by:
    \begin{equation} 
\begin{array}{rcl}
 \tilde{g}^{c}((0,Y),(0,Y^{'}))  &=& 0,\\
\tilde{g}^{c}((X,0), (0,Y^{'})) &=& (\tilde{g}(X,Y^{'}))^{v} \\
\tilde{g}^{c}((X,0), (X^{'},0)) &=& (\tilde{g}(X,X^{'}))^{c}.
\end{array}
\end{equation}
The left invariant contravariant pseudo-Riemannian metric $g^{c}$  on $TG$ associated to 
$\tilde{g}^{c}$ is given for any $(\alpha,\beta), (\alpha^{'},\beta^{'}) \in \mathfrak{g}^{*} \times \mathfrak{g}^{*}$ by:
     \begin{equation} \label{e3}
\begin{array}{rcl}
 g^{c}((\alpha,0),(\alpha^{'},0))  &=& 0,\\
g^{c}((\alpha,0), (0,\beta^{'})) &=& (g(\alpha, \beta^{'}))^{v} \\
g^{c}((0,\beta), (0,\beta^{'})) &=& (g(\beta, \beta^{'}))^{c}.
\end{array}
\end{equation}   
We call this metric $g^{c}$ the complete left invariant contravariant pseudo-Riemannian metric on $TG$.\\~~~~\\
In the following proposition, we express the Levi-Civita contravariant connection $\mathcal{D}^{TG}$ associated to  $(\Pi_{TG},g^{c})$ in term of the Levi-Civita contravariant connection $\mathcal{D}^{G}$ associated to $(\Pi_{G},g)$.
     \begin{prop} \label{cnxcont}
     Let $(G,\Pi_{G},g)$ be a Poisson-Lie group equipped with the left invariant contravariant pseudo-Riemannian metric $g$ and $(TG,\Pi_{TG},g^{c})$  the Sanchez de Alvarez tangent Poisson-Lie group of $G$
     equipped with the complete left invariant  pseudo-Riemannian metric $g^{c}$ associated to $g$. Let $\mathcal{D}^{TG}$ and $\mathcal{D}^{G}$ be the Levi-Civita contravariant connections associated respectively to $(\Pi_{TG},g^{c})$ and $(\Pi_{G},g)$.
Then for any $(\alpha,\beta), (\alpha^{'},\beta^{'}) \in \mathfrak{g}^{*} \times \mathfrak{g}^{*}$ we have :$$\mathcal{D}^{TG}_{(\alpha,\beta)}(\alpha^{'},\beta^{'}) = \Big(\mathcal{D}^{G}_{\alpha}\beta^{'} + \mathcal{D}^{G}_{\beta}\alpha^{'} ,\mathcal{D}^{G}_{\beta}\beta^{'} \Big).$$
     \end{prop}
     \begin{proof}
     According to equations (\ref{e1}), (\ref{e2}) and (\ref{e3}) we obtain:
     $$ \begin{array}{rcl}
 2g^{c}(\mathcal{D}^{TG}_{(\alpha,\beta)}(\alpha^{'},\beta^{'}), (\alpha^{''},\beta^{''})) &=& g^{c}([(\alpha,\beta),(\alpha^{'},\beta^{'})]_{\mathfrak{g}^{*} \ltimes \mathfrak{g}^{*}},(\alpha^{''},\beta^{''})) + 
 g^{c}([(\alpha^{''},\beta^{''}),(\alpha,\beta)]_{\mathfrak{g}^{*} \ltimes \mathfrak{g}^{*}},(\alpha^{'},\beta^{'}))\\ &\quad \quad & + \quad g^{c}([(\alpha^{''},\beta^{''}),(\alpha^{'},\beta^{'})]_{\mathfrak{g}^{*} \ltimes \mathfrak{g}^{*}},(\alpha,\beta))\\ &=& g^{c}(([\alpha,\beta^{'}]_{\mathfrak{g}^{*}}
 + [\beta,\alpha^{'}]_{\mathfrak{g}^{*}},[\beta,\beta^{'}]_{\mathfrak{g}^{*}}),(\alpha^{''},\beta^{''}))\\  &\quad \quad & + \quad  g^{c}(([\alpha^{''},\beta]_{\mathfrak{g}^{*}}
 + [\beta^{''},\alpha]_{\mathfrak{g}^{*}},[\beta^{''},\beta]_{\mathfrak{g}^{*}}),(\alpha^{'},\beta^{'}))\\ &\quad \quad & + \quad  g^{c}(([\alpha^{''},\beta^{'}]_{\mathfrak{g}^{*}}
 + [\beta^{''},\alpha^{'}]_{\mathfrak{g}^{*}},[\beta^{''},\beta^{'}]_{\mathfrak{g}^{*}}),(\alpha,\beta))\\
 &=& \Big( g([\alpha,\beta^{'}]_{\mathfrak{g}^{*}},\beta^{''}) + g([\beta,\alpha^{'}]_{\mathfrak{g}^{*}},\beta^{''}) + g([\beta,\beta^{'}]_{\mathfrak{g}^{*}},\alpha^{''})\Big)^{v} \\&\quad \quad & + \quad    
 \Big( g([\alpha^{''},\beta]_{\mathfrak{g}^{*}},\beta^{'}) + g([\beta^{''},\alpha]_{\mathfrak{g}^{*}},\beta^{'}) + g([\beta^{''},\beta]_{\mathfrak{g}^{*}},\alpha^{'})\Big)^{v} \\ &\quad \quad &  + \quad
 \Big( g([\alpha^{''},\beta^{'}]_{\mathfrak{g}^{*}},\beta) + g([\beta^{''},\alpha^{'}]_{\mathfrak{g}^{*}},\beta) + g([\beta^{''},\beta^{'}]_{\mathfrak{g}^{*}},\alpha)\Big)^{v}\\ &\quad \quad & + \quad \Big(   g([\beta,\beta^{'}]_{\mathfrak{g}^{*}},\beta^{''}) + g([\beta^{''},\beta]_{\mathfrak{g}^{*}},\beta^{'}) + g[\beta^{''},\beta^{'}]_{\mathfrak{g}^{*}},\beta)\Big)^{c}
 \\ &=& 2\Big((g(\mathcal{D}^{G}_{\alpha}\beta^{'},\beta^{''}))^{v} + (g(\mathcal{D}^{G}_{\beta}\alpha^{'},\beta^{''}))^{v} +   (g(\mathcal{D}^{G}_{\beta}\beta^{'},\alpha^{''}))^{v} + (g(\mathcal{D}^{G}_{\beta}\beta^{'},\beta^{''}))^{c} \Big) \\ &=& 2 g^{c}((\mathcal{D}^{G}_{\alpha}\beta^{'} + \mathcal{D}^{G}_{\beta}\alpha^{'},\mathcal{D}^{G}_{\beta}\beta^{'}), (\alpha^{''},\beta^{''})).
 \end{array}$$
     \end{proof}
     \begin{lem} \label{courbure}
      Let $R^{TG}$ and $R^{G}$ be the curvatures of $\mathcal{D}^{TG}$ and  $\mathcal{D}^{G}$ respectively. Then for any $(\alpha,\beta), (\alpha^{'},\beta^{'}), (\alpha^{''},\beta^{''}) \in \mathfrak{g}^{*} \times \mathfrak{g}^{*}$ we have :
    $$\mathcal{R}^{TG}((\alpha,\beta),(\alpha^{'},\beta^{'}))(\alpha^{''},\beta^{''}) = \Big(\mathcal{R}^{G}(\alpha,\beta^{'})\beta^{''} + \mathcal{R}^{G}(\beta,\alpha^{'})\beta^{''} + \mathcal{R}^{G}(\beta,\beta^{'})\alpha^{''},\mathcal{R}^{G}(\beta,\beta^{'})\beta^{''} \Big).$$
     \end{lem}
     \begin{proof}
     Using the definition of the curvature tensor (\ref{curva}) and  proposition (\ref{cnxcont}). We obtain: 
      $$\begin{array}{rcl}
\mathcal{R}^{TG}((\alpha,\beta),(\alpha^{'},\beta^{'}))(\alpha^{''},\beta^{''}) &=& \mathcal{D}^{TG}_{(\alpha,\beta)}\mathcal{D}^{TG}_{(\alpha^{'},\beta^{'})}(\alpha^{''},\beta^{''}) - \mathcal{D}^{TG}_{(\alpha^{'},\beta^{'})}\mathcal{D}^{TG}_{(\alpha,\beta)}(\alpha^{''},\beta^{''})\\ &\quad \quad & - \quad  \mathcal{D}^{TG}_{[(\alpha,\beta),(\alpha^{'},\beta^{'}]_{{\mathfrak{g}}^*\ltimes {\mathfrak{g}}^*}}(\alpha^{''},\beta^{''}) \\ &=& \mathcal{D}^{TG}_{(\alpha,\beta)}(\mathcal{D}^{G}_{\alpha^{'}}\beta^{''} + \mathcal{D}^{G}_{\beta^{'}}\alpha^{''},\mathcal{D}^{G}_{\beta^{'}}\beta^{''}) - \mathcal{D}^{TG}_{(\alpha^{'},\beta^{'})}(\mathcal{D}^{G}_{\alpha}\beta^{''} + \mathcal{D}^{G}_{\beta}\alpha^{''},\mathcal{D}^{G}_{\beta}\beta^{''})\\ &\quad \quad & - \quad \mathcal{D}^{TG}_{([\alpha,\beta^{'}]_{\mathfrak{g}^{*}} + [\beta,\alpha^{'}]_{\mathfrak{g}^{*}},[\beta,\beta^{'}]_{\mathfrak{g}^{*}})}(\alpha^{''},\beta^{''})\\ &=&
\Big(\mathcal{D}^{G}_{\alpha}\mathcal{D}^{G}_{\beta^{'}}\beta^{''} + \mathcal{D}^{G}_{\beta}\mathcal{D}^{G}_{\alpha^{'}}\beta^{''} + \mathcal{D}^{G}_{\beta}\mathcal{D}^{G}_{\beta^{'}}\alpha^{''},\mathcal{D}^{G}_{\beta}\mathcal{D}^{G}_{\beta^{'}}\beta^{''}\Big)\\ &\quad \quad & - \quad 
\Big(\mathcal{D}^{G}_{\alpha^{'}}\mathcal{D}^{G}_{\beta}\beta^{''} + \mathcal{D}^{G}_{\beta^{'}}\mathcal{D}^{G}_{\alpha}\beta^{''} + \mathcal{D}^{G}_{\beta^{'}}\mathcal{D}^{G}_{\beta}\alpha^{''},\mathcal{D}^{G}_{\beta^{'}}\mathcal{D}^{G}_{\beta}\beta^{''}\Big)\\ &\quad \quad & - \quad
\Big(\mathcal{D}^{G}_{[\alpha,\beta^{'}]_{\mathfrak{g}^{*}}}\beta^{''} + \mathcal{D}^{G}_{[\beta,\alpha^{'}]_{\mathfrak{g}^{*}}}\beta^{''} + \mathcal{D}^{G}_{[\beta,\beta^{'}]_{\mathfrak{g}^{*}}}\alpha^{''},\mathcal{D}^{G}_{[\beta,\beta^{'}]_{\mathfrak{g}^{*}}}\beta^{''}\Big)
 \\ &=& \Big(\mathcal{R}^{G}(\alpha,\beta^{'})\beta^{''} + \mathcal{R}^{G}(\beta,\alpha^{'})\beta^{''} + 
 \mathcal{R}^{G}(\beta,\beta^{'})\alpha^{''}, \mathcal{R}^{G}(\beta,\beta^{'})\beta^{''} \Big)
\end{array}$$
     \end{proof}
     \begin{lem} \label{sym}
      For any $(\alpha,\beta), (\alpha^{'},\beta^{'}), (\alpha^{''},\beta^{''}), (\alpha^{'''},\beta^{'''}) \in \mathfrak{g}^{*} \times \mathfrak{g}^{*}$ we have:
     $$ \begin{array}{rcl}
      (\mathcal{D}^{TG}_{(\alpha,\beta)}\mathcal{R}^{TG})((\alpha^{'},\beta^{'}),(\alpha^{''},\beta^{''})(\alpha^{'''},\beta^{'''}) &=& \Big((\mathcal{D}^{G}_{\alpha}\mathcal{R}^{G})(\beta^{'},\beta^{''})\beta^{'''} + (\mathcal{D}^{G}_{\beta}\mathcal{R}^{G})(\alpha^{'},\beta^{''})\beta^{'''}\\  &+&  (\mathcal{D}^{G}_{\beta}\mathcal{R}^{G})(\beta^{'},\alpha^{''})\beta^{'''} +  (\mathcal{D}^{G}_{\beta}\mathcal{R}^{G})(\beta^{'},\beta^{''})\alpha^{'''}, (\mathcal{D}^{G}_{\beta}\mathcal{R}^{G})(\beta^{'},\beta^{''})\beta^{'''}\Big)
      \end{array}$$
     \end{lem}
     \begin{proof}
     According to equation (\ref{local}), proposition (\ref{cnxcont}) and lemma (\ref{courbure}) we obtain:
     \begin{equation*}
     \begin{array}{rcl}
     (\mathcal{D}^{TG}_{(\alpha,\beta)}\mathcal{R}^{TG}) &=& \mathcal{D}^{TG}_{(\alpha,\beta)}(\mathcal{R}^{TG}((\alpha^{'},\beta^{'}),(\alpha^{''},\beta^{''}))(\alpha^{'''},\beta^{'''}) - \mathcal{R}^{TG}(\mathcal{D}^{TG}_{(\alpha,\beta)}(\alpha^{'},\beta^{'}),(\alpha^{''},\beta^{''}))(\alpha^{'''},\beta^{'''})\\ &\quad \quad & - \quad \mathcal{R}^{TG}((\alpha^{'},\beta^{'}),(\alpha^{''},\beta^{''}))\mathcal{D}^{TG}_{(\alpha,\beta)}(\alpha^{'''},\beta^{'''}) - \mathcal{R}^{TG}((\alpha^{'},\beta^{'}),\mathcal{D}^{TG}_{(\alpha,\beta)}(\alpha^{''},\beta^{''}))(\alpha^{'''},\beta^{'''})\\ &=&
     \mathcal{D}^{TG}_{(\alpha,\beta)} \Big((\mathcal{R}^{G}(\alpha^{'},\beta^{''})\beta^{'''} + \mathcal{R}^{G}(\beta^{'},\alpha^{''})\beta^{'''} + \mathcal{R}^{G}(\beta^{'},\beta^{''})\alpha^{'''},\mathcal{R}^{G}(\beta^{'},\beta^{''})\beta^{'''} \Big)\\ &\quad \quad & - \quad
     \mathcal{R}^{TG}((\mathcal{D}^{G}_{\alpha}\beta^{'} + \mathcal{D}^{G}_{\beta}\alpha^{'},\mathcal{D}^{G}_{\beta}\beta^{'}),(\alpha^{''},\beta^{''}))(\alpha^{'''},\beta^{'''}) \\ &\quad \quad & - \quad
     \mathcal{R}^{TG}((\alpha^{'},\beta^{'}),(\alpha^{''},\beta^{''}))(\mathcal{D}^{G}_{\alpha}\beta^{'''} + \mathcal{D}^{G}_{\beta}\alpha^{'''},\mathcal{D}^{G}_{\beta}\beta^{'''}) \\ &\quad \quad & - \quad
     \mathcal{R}^{TG}((\alpha^{'},\beta^{'}),(\mathcal{D}^{G}_{\alpha}\beta^{''} + \mathcal{D}^{G}_{\beta}\alpha^{''},\mathcal{D}^{G}_{\beta}\beta^{''}))(\alpha^{'''},\beta^{'''})\\ &=&
     \Big(\mathcal{D}^{G}_{\alpha}(\mathcal{R}^{G}(\beta^{'},\beta^{''})\beta^{'''}) + \mathcal{D}^{G}_{\beta}(\mathcal{R}^{G}(\alpha^{'},\beta^{''})\beta^{'''}) + \mathcal{D}^{G}_{\beta}(\mathcal{R}^{G}(\beta^{'},\alpha^{''})\beta^{'''})\\ &\quad \quad & + \quad \mathcal{D}^{G}_{\beta}(\mathcal{R}^{G}(\beta^{'},\beta^{''})\alpha^{'''}), \mathcal{D}^{G}_{\beta}(\mathcal{R}^{G}(\beta^{'},\beta^{''})\beta^{'''}) \Big) -
     \Big(\mathcal{R}^{G}(\mathcal{D}^{G}_{\alpha}\beta^{'},\beta^{''})\beta^{'''} \\ &\quad \quad & + \quad \mathcal{R}^{G}(\mathcal{D}^{G}_{\beta}\alpha^{'},\beta^{''})\beta^{'''} + \mathcal{R}^{G}(\mathcal{D}^{G}_{\beta}\beta^{'},\alpha^{''})\beta^{'''} + \mathcal{R}^{G}(\mathcal{D}^{G}_{\beta}\beta^{'},\beta^{''})\alpha^{'''},\mathcal{R}^{G}(\mathcal{D}^{G}_{\beta}\beta^{'},\beta^{''})\beta^{'''} \Big)\\ &\quad \quad & - \quad
     \Big(\mathcal{R}^{G}(\alpha^{'},\beta^{''})\mathcal{D}^{G}_{\beta}\beta^{'''} + \mathcal{R}^{G}(\beta^{'},\alpha^{''})\mathcal{D}^{G}_{\beta}\beta^{'''} + \mathcal{R}^{G}(\beta^{'},\beta^{''})\mathcal{D}^{G}_{\alpha}\beta^{'''}\\ &\quad \quad & + \quad \mathcal{R}^{G}(\beta^{'},\beta^{''})\mathcal{D}^{G}_{\beta}\alpha^{'''},\mathcal{R}^{G}(\beta^{'},\beta^{''})\mathcal{D}^{G}_{\beta}\beta^{'''} \Big) - \Big(\mathcal{R}^{G}(\alpha^{'},\mathcal{D}^{G}_{\beta}\beta^{''})\beta^{'''} + \mathcal{R}^{G}(\beta^{'},\mathcal{D}^{G}_{\alpha}\beta^{''})\beta^{'''} \\ &\quad \quad & + \quad \mathcal{R}^{G}(\beta^{'},\mathcal{D}^{G}_{\beta}\alpha^{''})\beta^{'''} + \mathcal{R}^{G}(\beta^{'},\mathcal{D}^{G}_{\beta}\beta^{''})\alpha^{'''},\mathcal{R}^{G}(\beta^{'},\mathcal{D}^{G}_{\beta}\beta^{'})\beta^{'''} \Big)\\ &=& \Big((\mathcal{D}^{G}_{\alpha}\mathcal{R}^{G})(\beta^{'},\beta^{''})\beta^{'''} + (\mathcal{D}^{G}_{\beta}\mathcal{R}^{G})(\alpha^{'},\beta^{''})\beta^{'''} + (\mathcal{D}^{G}_{\beta}\mathcal{R}^{G})(\beta^{'},\alpha^{''})\beta^{'''}\\ &\quad \quad & + \quad (\mathcal{D}^{G}_{\beta}\mathcal{R}^{G})(\beta^{'},\beta^{''})\alpha^{'''}, (\mathcal{D}^{G}_{\beta}\mathcal{R}^{G})(\beta^{'},\beta^{''})\beta^{'''}\Big).
     \end{array}
     \end{equation*}
     \end{proof}
      \begin{thm} \label{th1}
       Let $(G,\Pi_{G},g)$ be a Poisson-Lie group  equipped with the left invariant contravariant pseudo-Riemannian metric $g$ and $(TG,\Pi_{TG},g^{c})$ be the  Sanchez de Alvarez tangent Poisson-Lie group of $G$
     equipped with the left invariant complete  pseudo-Riemannian metric $g^{c}$ associated to $g$. Let $\mathcal{D}^{TG}$ and $\mathcal{D}^{G}$ be the Levi-Civita contravariant connections associated respectively to $(\Pi_{TG},g^{c})$ and $(\Pi_{G},g)$. Then
     \begin{enumerate}
     \item  $\mathcal{D}^{G}$ is flat if, and only if,  $\mathcal{D}^{TG}$ is flat.
     
     \item   $\mathcal{D}^{G}$ is locally symmetric if and only if $\mathcal{D}^{TG}$ is  locally symmetric.
     \end{enumerate}
      \end{thm}
      \begin{proof}
      \begin{enumerate}
      \item It is a direct consequence of the lemma (\ref{courbure}) that if $\mathcal{R}^{G} = 0,$ then $\mathcal{R}^{TG} = (0,0).$ We now assume that $\mathcal{R}^{TG} = (0,0)$ then for any $\beta,\beta^{'},\beta^{''} \in \mathfrak{g}^{*}$ wa have $$\mathcal{R}^{G}(\beta,\beta^{'})\beta^{''} = 0.$$
      Therefore $\mathcal{D}^{G}$ is flat.
      \item According to lemma (\ref{sym}), if  $\mathcal{D}^{G}\mathcal{R}^{G} = 0,$ then $\mathcal{D}^{TG}\mathcal{R}^{TG} = (0,0).$ Conversely, if $\mathcal{D}^{TG}\mathcal{R}^{TG} = (0,0)$ then for any 
      $\beta,\beta^{'},\beta^{''},\beta^{'''} \in \mathfrak{g}^{*}$ wa have $$\mathcal{D}^{G}_{\beta}\mathcal{R}^{G}(\beta^{'},\beta^{''})\beta^{'''} = 0.$$
      Hence $\mathcal{D}^{G}$ is locally symmetric.

     \end{enumerate}
      \end{proof}
       \begin{lem} \label{generalized koszul}
       Let $[ , ]_{\mathfrak{g}^{*} \ltimes \mathfrak{g}^{*}}$ and $[ , ]_{\mathfrak{g}^{*}}$ be the  generalized Koszul brackets  on $\Omega^{*}(TG)$ and  $\Omega^{*}(G)$ respectively. Then for any
       $(\alpha,\beta), (\alpha^{'},\beta^{'}) \in \mathfrak{g}^{*} \times \mathfrak{g}^{*}$ we have:
       $$[(\alpha,\beta),d(\alpha^{'},\beta^{'})]_{\mathfrak{g}^{*} \ltimes \mathfrak{g}^{*}} = ([\alpha,d\beta^{'}]_{\mathfrak{g}^{*}} + [\beta,d\alpha^{'}]_{\mathfrak{g}^{*}},[\beta,d\beta^{'}]_{\mathfrak{g}^{*}}).$$
       \end{lem}
       \begin{proof}
        Let $(x_{i})$ are local coordinates of $G$ in a neighborhood of  $e$. We write  $\beta^{'} = \sum \limits_{i} \beta_{i}^{'}dx_{i}$ and $\alpha^{'} = \sum \limits_{i} \alpha_{i}^{'}dx_{i}$. Then using equations
       (\ref{id}), (\ref{id1}) and (\ref{e2}) we obtain:
        $$\begin{array}{rcl}
       [(0,\beta),d(0,\beta^{'})]_{\mathfrak{g}^{*} \ltimes \mathfrak{g}^{*}} &=& \sum \limits_{i} [(0,\beta),d(0,\beta_{i}^{'}dx_{i})]_{\mathfrak{g}^{*} \ltimes \mathfrak{g}^{*}} = \sum \limits_{i} [(0,\beta),(0,d\beta_{i}^{'} \wedge dx_{i})]_{\mathfrak{g}^{*} \ltimes \mathfrak{g}^{*}}\\ &=& \sum \limits_{i} \Big( [(0,\beta),(0,d\beta_{i}^{'}) \wedge (dx_{i},0) + (d\beta_{i}^{'},0) \wedge (0,dx_{i})]_{\mathfrak{g}^{*} \ltimes \mathfrak{g}^{*}} \Big)\\
       &=& \sum \limits_{i} \Big([(0,\beta),(0,d\beta_{i}^{'})]_{\mathfrak{g}^{*} \ltimes \mathfrak{g}^{*}} \wedge (dx_{i},0) + (0,d\beta_{i}^{'}) \wedge [(0,\beta), (dx_{i},0)]_{\mathfrak{g}^{*} \ltimes \mathfrak{g}^{*}}\\ &\quad \quad & + \quad 
       [(0,\beta),(d\beta_{i}^{'},0)]_{\mathfrak{g}^{*} \ltimes \mathfrak{g}^{*}} \wedge (0,dx_{i}) + (d\beta_{i}^{'},0) \wedge [(0,\beta), (0,dx_{i})]_{\mathfrak{g}^{*} \ltimes \mathfrak{g}^{*}} \Big)
       \\ &=& \sum \limits_{i} \Big( (0,[\beta,d\beta_{i}^{'}]_{\mathfrak{g}^{*}}) \wedge (dx_{i},0) + (0,d\beta_{i}^{'}) \wedge ([\beta, dx_{i}]_{\mathfrak{g}^{*}},0) \\ &\quad \quad & + \quad 
       ([\beta,d\beta_{i}^{'}]_{\mathfrak{g}^{*}},0) \wedge (0,dx_{i}) + (d\beta_{i}^{'},0) \wedge (0,[\beta, dx_{i}]_{\mathfrak{g}^{*}})\Big)
       \\&=& \sum \limits_{i} \Big((0,[\beta,d\beta_{i}^{'}]_{\mathfrak{g}^{*}} \wedge dx_{i})+ (0,d\beta_{i}^{'} \wedge [\beta, dx_{i}]_{\mathfrak{g}^{*}}) \Big)\\ &=& \sum \limits_{i}(0,[\beta,d\beta_{i}^{'} \wedge dx_{i}]_{\mathfrak{g}^{*}}) \\&=& (0,[\beta,d\beta^{'}]_{\mathfrak{g}^{*}}).
       \end{array}$$
         $$\begin{array}{rcl}
       [(0,\beta),d(\alpha^{'},0)]_{\mathfrak{g}^{*} \ltimes \mathfrak{g}^{*}} &=& \sum \limits_{i} [(0,\beta),(d\alpha_{i}^{'} \wedge dx_{i},0)]_{\mathfrak{g}^{*} \ltimes \mathfrak{g}^{*}} = \sum \limits_{i} [(0,\beta),(d\alpha_{i}^{'},0) \wedge (dx_{i},0)]_{\mathfrak{g}^{*} \ltimes \mathfrak{g}^{*}}\\
       &=& \sum \limits_{i} \Big([(0,\beta),(d\alpha_{i}^{'},0)]_{\mathfrak{g}^{*} \ltimes \mathfrak{g}^{*}} \wedge (dx_{i},0) + (d\alpha_{i}^{'},0) \wedge [(0,\beta), (dx_{i},0)]_{\mathfrak{g}^{*} \ltimes \mathfrak{g}^{*}} \Big)
       \\ &=& \sum \limits_{i} \Big(([\beta,d\alpha_{i}^{'}]_{\mathfrak{g}^{*}},0) \wedge (dx_{i},0) + (d\alpha_{i}^{'},0) \wedge ([\beta, dx_{i}]_{\mathfrak{g}^{*}},0) \Big) 
       \\ &=& \sum \limits_{i} ([\beta,d \alpha_{i}^{'} \wedge dx_{i}]_{\mathfrak{g}^{*}},0) \\&=& ([\beta,d\alpha^{'}]_{\mathfrak{g}^{*}},0).
       \end{array}$$
         $$\begin{array}{rcl}
       [(\alpha,0),d(0,\beta^{'})]_{\mathfrak{g}^{*} \ltimes \mathfrak{g}^{*}} &=& \sum \limits_{i} [(\alpha,0),(0,d\beta_{i}^{'} \wedge dx_{i})]_{\mathfrak{g}^{*} \ltimes \mathfrak{g}^{*}}\\ &=& \sum \limits_{i} [(\alpha,0),(0,d\beta_{i}^{'}) \wedge (dx_{i},0) + (d\beta_{i}^{'},0) \wedge (0,dx_{i})]_{\mathfrak{g}^{*} \ltimes \mathfrak{g}^{*}}\\
       &=& \sum \limits_{i} \Big([(\alpha,0),(0,d\beta_{i}^{'})]_{\mathfrak{g}^{*} \ltimes \mathfrak{g}^{*}} \wedge (dx_{i},0) + (0,d\beta_{i}^{'}) \wedge [(\alpha,0), (dx_{i},0)]_{\mathfrak{g}^{*} \ltimes \mathfrak{g}^{*}}\\ &\quad \quad & + \quad  
       [(\alpha,0),(d\beta_{i}^{'},0)]_{\mathfrak{g}^{*} \ltimes \mathfrak{g}^{*}} \wedge (0,dx_{i}) + (d\beta_{i}^{'},0) \wedge [(\alpha,0), (0,dx_{i})]_{\mathfrak{g}^{*} \ltimes \mathfrak{g}^{*}} \Big)
       \\ &=& \sum \limits_{i} \Big(([\alpha,d\beta_{i}^{'}]_{\mathfrak{g}^{*}},0) \wedge (dx_{i},0) + (d\beta_{i}^{'},0) \wedge ([\alpha, dx_{i}]_{\mathfrak{g}^{*}},0) \Big)
        \\&=& ([\alpha,d\beta^{'}]_{\mathfrak{g}^{*}},0).
       \end{array}$$
       $$\begin{array}{rcl}
       [(\alpha,0),d(\alpha^{'},0)]_{\mathfrak{g}^{*} \ltimes \mathfrak{g}^{*}} &=& \sum \limits_{i}  [(\alpha,0),(d\alpha_{i}^{'} \wedge dx_{i},0)]_{\mathfrak{g}^{*} \ltimes \mathfrak{g}^{*}} = \sum \limits_{i}  [(\alpha,0),(d\alpha_{i}^{'},0) \wedge (dx_{i},0)]_{\mathfrak{g}^{*} \ltimes \mathfrak{g}^{*}}\\
       &=& \sum \limits_{i} \Big( [(\alpha,0),(d\alpha_{i}^{'},0)]_{\mathfrak{g}^{*} \ltimes \mathfrak{g}^{*}} \wedge (dx_{i},0) + (d\alpha_{i}^{'},0) \wedge [(\alpha,0), (dx_{i},0)]_{\mathfrak{g}^{*} \ltimes \mathfrak{g}^{*}} \Big)
       \\ &=& 0.
       \end{array}$$
       \end{proof}
       \begin{prop} \label{generalized}
       Let $\{ , \}_{TG}$ and $\{ , \}_{G}$ be the Hawkins generalized pre-Poisson brackets  of the Levi-Civita contravariant connections $\mathcal{D}^{TG}$ and  $\mathcal{D}^{G}$ respectively.
       Then for any $(\alpha,\beta),(\alpha^{'},\beta^{'}) \in \mathfrak{g}^{*} \times \mathfrak{g}^{*}$ we have :
      $$\{(\alpha,\beta),(\alpha^{'},\beta^{'})\}_{TG} = (\{\alpha,\beta^{'}\}_{G} + \{\beta,\alpha^{'}\}_{G}, \{\beta,\beta^{'}\}_{G} ).$$
       \end{prop}
       \begin{proof}
      Note that the Levi-Civita contravariant connections $\mathcal{D}^{G}$ and $\mathcal{D}^{TG}$ naturally extends to $\Omega^{2}(G)$ and $\Omega^{2}(TG)$ respectively.
       Using equation (\ref{bracket}), proposition (\ref{cnxcont}) and lemma (\ref{generalized koszul}) we obtain,
        $$ \begin{array}{rcl}
     \{(\alpha,\beta),(\alpha^{'},\beta^{'})\}_{TG} &=& -\mathcal{D}^{TG}_{(\alpha,\beta)}d(\alpha^{'},\beta^{'}) -\mathcal{D}^{TG}_{(\alpha^{'},\beta^{'})}d(\alpha,\beta) + d \mathcal{D}^{TG}_{(\alpha^{'},\beta^{'})}(\alpha,\beta) + [(\alpha,\beta),d(\alpha^{'},\beta^{'})]_{\mathfrak{g}^{*} \ltimes \mathfrak{g}^{*}} \\ &=& -(\mathcal{D}^{G}_{\alpha}d\beta^{'} + \mathcal{D}^{G}_{\beta}d\alpha^{'},\mathcal{D}^{G}_{\beta}d\beta^{'}) -(\mathcal{D}^{G}_{\alpha^{'}}d\beta + \mathcal{D}^{G}_{\beta^{'}}d\alpha, \mathcal{D}^{G}_{\beta^{'}}d\beta) \\ &\quad \quad & + \quad (d\mathcal{D}^{G}_{\alpha^{'}}\beta + d\mathcal{D}^{G}_{\beta^{'}}\alpha,d\mathcal{D}^{G}_{\beta^{'}}\beta) + ([\alpha,d\beta^{'}]_{\mathfrak{g}^{*}} + [\beta,d\alpha^{'}]_{\mathfrak{g}^{*}},[\beta,d\beta^{'}]_{\mathfrak{g}^{*}}) \\ &=& \Big(-\mathcal{D}^{G}_{\alpha}d\beta^{'} - \mathcal{D}^{G}_{\beta}d\alpha^{'} - \mathcal{D}^{G}_{\alpha^{'}}d\beta - \mathcal{D}^{G}_{\beta^{'}}d\alpha + d\mathcal{D}^{G}_{\alpha^{'}}\beta + d\mathcal{D}^{G}_{\beta^{'}}\alpha\\ &\quad \quad & + \quad  [\alpha,d\beta^{'}]_{\mathfrak{g}^{*}} + [\beta,d\alpha^{'}]_{\mathfrak{g}^{*}} , -\mathcal{D}^{G}_{\beta}d\beta^{'} -\mathcal{D}^{G}_{\beta^{'}}d\beta + d\mathcal{D}^{G}_{\beta^{'}}\beta + [\beta,d\beta^{'}]_{\mathfrak{g}^{*}} \Big) \\ &=& \Big(\{\alpha,\beta^{'}\}_{G} + \{\beta,\alpha^{'}\}_{G}, \{\beta,\beta^{'}\}_{G} \Big).
    \end{array}$$
       \end{proof}
       \begin{lem} \label{metacurv}
     Let $\mathcal{M}^{G}$ and $\mathcal{M}^{TG}$ be the metacurvatures of the Levi-Civita contravariant connections $\mathcal{D}^{G}$ and $\mathcal{D}^{TG}$ respectively.
     Then for any $(\alpha,\beta),(\alpha^{'},\beta^{'}),(\alpha^{''},\beta^{''}) \in \mathfrak{g}^{*} \times \mathfrak{g}^{*}$ we have:
     \begin{enumerate}
     \item $\mathcal{M}^{TG}((\alpha,0),(\alpha^{'},0),(\alpha^{''},0)) = 0,$
     \item $\mathcal{M}^{TG}((\alpha,0),(\alpha^{'},0),(0,\beta^{''})) = 0,$
     \item  $ \mathcal{M}^{TG}((\alpha,0),(0,\beta^{'}),(0,\beta^{''}))   = (\mathcal{M}^{G}(\alpha,\beta^{'},\beta^{''}),0)$
     \item $\mathcal{M}^{TG}((0,\beta),(0,\beta^{'}),(0,\beta^{''})) = (0,\mathcal{M}^{G}(\beta,\beta^{'},\beta^{''})),$
     \item $\mathcal{M}^{TG}((0,\beta),(0,\beta^{'}),(\alpha^{''},0)) = (\mathcal{M}^{G}(\beta,\beta^{'},\alpha^{''}),0),$
     \item $\mathcal{M}^{TG}((0,\beta),(\alpha^{'},0),(\alpha^{''},0)) = 0.$
     \end{enumerate}
         
       \end{lem}
       \begin{proof}
      Note that the tangent Lie group $TG$ is isomorphic to $G \times \mathfrak{g}$. Then we denote by 
       $\pi : G \times \mathfrak{g} \rightarrow G ; (x,Y) \mapsto x$ and $\pi_{1} : G \times \mathfrak{g} \rightarrow \mathfrak{g} ; (x,Y) \mapsto Y$ the projection maps.  Let $\pi^{*} : \mathfrak{g}^{*}  \rightarrow \mathfrak{g}^{*} \times \mathfrak{g}^{*} ; \alpha \mapsto (\alpha,0)$ and $\pi_{1}^{*} : \mathfrak{g}^{*} \rightarrow \mathfrak{g}^{*} \times \mathfrak{g}^{*} ; \beta \mapsto (0,\beta)$ the transposed of tangent maps at point $(e,0)$ of $\pi$ and $\pi_{1}$ respectively.\\
       Using equation (\ref{equa2.3}) and propositions (\ref{cnxcont}) and (\ref{generalized}), then locally for any $\alpha = \sum \limits_{i} \alpha_{i} dx_{i}$ and  $\beta = \sum \limits_{i} \beta_{i}  dx_{i}$ we obtain:
       \begin{enumerate}
       \item 
       $$\begin{array}{rcl}
       \mathcal{M}^{TG}((\alpha,0),(\alpha^{'},0),(\alpha^{''},0)) &=& \sum \limits_{i} \alpha_{i} \Big(\{ x_{i} \circ \pi,\{(\alpha^{'},0),(\alpha^{''},0)\}_{TG}\}_{TG}\\ &-& \{\{ x_{i} \circ \pi,(\alpha^{'},0)\}_{TG},(\alpha^{''},0)\}_{TG} - \{\{x_{i} \circ \pi,(\alpha^{''},0)\}_{TG},(\alpha^{'},0)\}_{TG}\Big)\\ &=& - \sum \limits_{i} \alpha_{i} \Big(\{\mathcal{D}^{TG}_{(dx_{i},0)} (\alpha^{'},0),(\alpha^{''},0)\}_{TG} - \{\mathcal{D}^{TG}_{(dx_{i},0)}(\alpha^{''},0),(\alpha^{'},0)\}_{TG} \Big) \\&=& 0.
       \end{array}$$
     \item   $$\begin{array}{rcl}
       \mathcal{M}^{TG}((\alpha,0),(\alpha^{'},0),(0,\beta^{''})) &=& \sum \limits_{i} \alpha_{i} \Big(\{ x_{i} \circ \pi,\{(\alpha^{'},0),(0,\beta^{''})\}_{TG}\}_{TG}\\ &-& \{\{ x_{i} \circ \pi,(\alpha^{'},0)\}_{TG},(0,\beta^{''})\}_{TG} - \{\{x_{i} \circ \pi_{1},(0,\beta^{''})\}_{TG},(\alpha^{'},0)\}_{TG}\Big)\\ &=& \sum \limits_{i} \alpha_{i} \Big( \mathcal{D}^{TG}_{(dx_{i},0)}(\{\alpha^{'},\beta^{''}\}_{G},0) - \{\mathcal{D}^{TG}_{(dx_{i},0)} (\alpha^{'},0),(0,\beta^{''})\}_{TG}\\ &\quad & - \quad  \{\mathcal{D}^{TG}_{(dx_{i},0)}(0,\beta^{''}),(\alpha^{'},0)\}_{TG} \Big)\\ &=& \sum \limits_{i} \alpha_{i} \Big( \{(\mathcal{D}^{G}_{dx_{i}}\beta^{''},0),(\alpha^{'},0)\}_{TG}\Big) \\&=& 0.
       \end{array}$$
       \item  $$\begin{array}{rcl}
       \mathcal{M}^{TG}((\alpha,0),(0,\beta^{'}),(0,\beta^{''})) &=& \sum \limits_{i} \alpha_{i} \Big(\{ x_{i} \circ \pi,\{(0,\beta^{'}),(0,\beta^{''})\}_{TG}\}_{TG}\\ &-& \{\{ x_{i} \circ \pi_{1},(0,\beta^{'})\}_{TG},(0,\beta^{''})\}_{TG} - \{\{x_{i} \circ \pi,(0,\beta^{''})\}_{TG},(0,\beta^{'})\}_{TG}\Big)\\ &=& \sum \limits_{i} \alpha_{i} \Big( \mathcal{D}^{TG}_{(dx_{i},0)}(0,\{\beta^{'},\beta^{''}\}_{G}) - \{\mathcal{D}^{TG}_{(dx_{i},0)} (0,\beta^{'}),(0,\beta^{''})\}_{TG}\\ &-& \{\mathcal{D}^{TG}_{(dx_{i},0)}(0,\beta^{''}),(0,\beta^{'})\}_{TG} \Big)\\ &=& \sum \limits_{i} \alpha_{i} \Big((\mathcal{D}^{G}_{dx_{i}}\{\beta^{'},\beta^{''}\}_{G},0) - (\{\mathcal{D}^{G}_{dx_{i}}\beta^{'},\beta^{''}\}_{G},	0) \\
&-& (\{\mathcal{D}^{G}_{dx_{i}}\beta^{''},\beta^{'}\}_{G},0) \Big)\\ &=&
    \sum \limits_{i} \alpha_{i} \Big((\{x_{i},\{\beta^{'},\beta^{''}\}_{G}\}_{G},0) - (\{x_{i},\beta^{'}\}_{G},\beta^{''} \}_{G},0)\\ &-& (\{x_{i},\beta^{''}\}_{G},\beta^{'}\}_{G},0) \Big)\\  &=& (\mathcal{M}^{G}(\alpha,\beta^{'},\beta^{''}),0).   
    \end{array}$$
     \item $$\begin{array}{rcl}
\mathcal{M}^{TG}((0,\beta),(0,\beta^{'}),(0,\beta^{''})) &=& \sum \limits_{i} \beta_{i} \Big(\{ x_{i} \circ \pi_{1}, \{ (0,\beta^{'}), (0,\beta^{''})\}_{TG}\}_{TG}\\
&\quad \quad & - \quad \{\{x_{i} \circ \pi_{1} ,(0,\beta^{'}) \}_{TG} , (0,\beta^{''})\}_{TG}\\ &\quad \quad & - \quad \{ \{x_{i} \circ \pi_{1} , (0,\beta^{''})\}_{TG}, (0,\beta^{'})\}_{TG}\Big)\\ &=& \sum \limits_{i} \beta_{i}\Big(\mathcal{D}^{TG}_{(0,dx_{i})} (0,\{\beta^{'},\beta^{''}\}_{G}) - \{\mathcal{D}^{TG}_{(0,dx_{i})}(0,\beta^{'}),(0,\beta^{''}) \}_{TG}\\ &\quad \quad & - \quad \{\mathcal{D}^{TG}_{(0,dx_{i})}(0,\beta^{''}),(0,\beta^{'}) \}_{TG} \Big)\\ &=& \sum \limits_{i} \beta_{i} \Big(
(0,\mathcal{D}^{G}_{dx_{i}}\{\beta^{'},\beta^{''}\}_{G}) - \{(0,\mathcal{D}^{G}_{dx_{i}}\beta^{'}),(0,\beta^{''}) \}_{TG}\\ &\quad \quad & - \quad
  \{(0,\mathcal{D}^{G}_{dx_{i}}\beta^{''}),(0,\beta^{'}) \}_{TG} \Big)\\ &=& \sum \limits_{i} \beta_{i} \Big(
  (0,\mathcal{D}^{G}_{dx_{i}}\{\beta^{'},\beta^{''}\}_{G}) - (0,\{\mathcal{D}^{G}_{dx_{i}}\beta^{'},\beta^{''}\})\\ &\quad \quad & - \quad (0,\{\mathcal{D}^{G}_{dx_{i}}\beta^{''},\beta^{'}\}_{G}) \Big)\\
&=& \sum \limits_{i} \beta_{i} \Big((0,\{ x_{i}, \{ \beta^{'}, \beta^{''}\}_{G}\}_{G}
 - \{\{x_{i} ,\beta^{'} \}_{G} , \beta^{''}\}_{G}\\ &\quad \quad & - \quad \{ \{x_{i} , \beta^{''}\}_{G}, \beta^{'}\}_{G}) \Big)\\ &=& \sum \beta_{i}(0,\mathcal{M}^{G}(dx_{i},\beta^{'},\beta^{''})) \\&=&  (0,\mathcal{M}^{G}(\beta,\beta^{'},\beta^{''})).
\end{array}$$
\item $$\begin{array}{rcl}
\mathcal{M}^{TG}((0,\beta),(0,\beta^{'}),(\alpha^{''},0)) &=& \sum \limits_{i} \beta_{i} \Big(\{ x_{i} \circ \pi_{1}, \{ (0,\beta^{'}), (\alpha^{''},0)\}_{TG}\}_{TG}\\
&\quad \quad & - \quad \{\{x_{i} \circ \pi_{1} ,(0,\beta^{'}) \}_{TG} , (\alpha^{''},0)\}_{TG}\\ &\quad \quad & - \quad \{ \{x_{i} \circ \pi_{1} , (\alpha^{''},0)\}_{TG}, (0,\beta^{'})\}_{TG}\Big)\\ &=& \sum \limits_{i} \beta_{i}\Big(\mathcal{D}^{TG}_{(0,dx_{i})} (\{\beta^{'},\alpha^{''}\}_{G},0) - \{\mathcal{D}^{TG}_{(0,dx_{i})}(0,\beta^{'}),(\alpha^{''},0) \}_{TG}\\ &\quad \quad & - \quad \{\mathcal{D}^{TG}_{(0,dx_{i})}(\alpha^{''},0),(0,\beta^{'}) \}_{TG} \Big)\\ &=& \sum \limits_{i} \beta_{i} \Big(
(\mathcal{D}^{G}_{dx_{i}}\{\beta^{'},\alpha^{''}\}_{G},0) - \{(0,\mathcal{D}^{G}_{dx_{i}}\beta^{'}),(\alpha^{''},0) \}_{TG}\\ &\quad \quad & - \quad
  \{(\mathcal{D}^{G}_{dx_{i}}\alpha^{''},0),(0,\beta^{'}) \}_{TG} \Big)\\ &=& \sum \limits_{i} \beta_{i} \Big(
  (\mathcal{D}^{G}_{dx_{i}}\{\beta^{'},\alpha^{''}\}_{G},0) - (\{\mathcal{D}^{G}_{dx_{i}}\beta^{'},\alpha^{''}\}_{G},0)\\ &\quad \quad & - \quad (\{\mathcal{D}^{G}_{dx_{i}}\alpha^{''},\beta^{'}\}_{G},0) \Big)\\
&=& \sum \beta_{i} \Big((\{ x_{i}, \{ \beta^{'}, \alpha^{''}\}_{G}\}_{G}
 - \{\{x_{i} ,\beta^{'} \}_{G} , \alpha^{''}\}_{G}\\ &\quad \quad & - \quad \{ \{x_{i} , \alpha^{''}\}_{G}, \beta^{'}\}_{G},0) \Big)\\ &=& \sum \limits_{i} \beta_{i}(\mathcal{M}^{G}(dx_{i},\beta^{'},\alpha^{''}),0) \\&=&  (\mathcal{M}^{G}(\beta,\beta^{'},\alpha^{''}),0).
\end{array}$$
\item $$\begin{array}{rcl}
\mathcal{M}^{TG}((0,\beta),(\alpha^{'},0),(\alpha^{''},0)) &=& \sum \limits_{i} \beta_{i} \Big(\{ x_{i} \circ \pi_{1}, \{ (\alpha^{'},0), (\alpha^{''},0)\}_{TG}\}_{TG}\\
&\quad \quad & - \quad \{\{x_{i} \circ \pi_{1} ,(\alpha^{'},0) \}_{TG} , (\alpha^{''},0)\}_{TG}\\ &\quad \quad & - \quad \{ \{x_{i} \circ \pi_{1} , (\alpha^{''},0)\}_{TG}, (\alpha^{'},0)\}_{TG}\Big)\\ &=& \sum \limits_{i} \beta_{i}\Big( - \{\mathcal{D}^{TG}_{(0,dx_{i})}(\alpha^{'},0),(\alpha^{''},0) \}_{TG}\\ &\quad \quad & - \quad \{\mathcal{D}^{TG}_{(0,dx_{i})}(\alpha^{''},0),(\alpha^{'},0) \}_{TG} \Big)\\ &=& \sum \limits_{i} \beta_{i} \Big(
 - \{(\mathcal{D}^{G}_{dx_{i}}\alpha^{'},0),(\alpha^{''},0) \}_{TG}\\ &\quad \quad & - \quad
  \{(\mathcal{D}^{G}_{dx_{i}}\alpha^{''},0),(\alpha^{'},0) \}_{TG} \Big)\\ &=& 0.
\end{array}$$
\end{enumerate}
       \end{proof}
       \begin{thm}
       Let $(G,\Pi_{G},g)$ be a Poisson-Lie group equipped with the left invariant contravariant pseudo-Riemannian metric $g$ and $(TG,\Pi_{TG},g^{c})$  the Sanchez de Alvarez tangent Poisson-Lie group of $G$
     equipped with the complete left invariant  pseudo-Riemannian metric $g^{c}$ associated to $g$.
   Then:
   \begin{enumerate}
   \item The Levi-Civita contravariant connection $\mathcal{D}^{G}$ is metaflat if, and only if, $\mathcal{D}^{TG}$ is metaflat.
   \item  the bracket $\{ , \}_{G}$ is a generalized Poisson bracket on $\Omega^{*}(M)$ if, and only if, 
       $\{ , \}_{TG}$ is a generalized Poisson bracket on $\Omega^{*}(TG)$.
       \end{enumerate}
      \end{thm}
      \begin{proof}
      \begin{enumerate}
    \item According to lemma (\ref{metacurv}) for any $(\alpha,\beta),(\alpha^{'},\beta^{'}),(\alpha^{''},\beta^{''}) \in \mathfrak{g}^{*} \times \mathfrak{g}^{*}$ we obtain:
    $$\mathcal{M}^{TG}((\alpha,\beta),(\alpha^{'},\beta^{'}),(\alpha^{''},\beta^{''})) = \Big(\mathcal{M}^{G}(\alpha,\beta^{'},\beta^{''}) + \mathcal{M}^{G}(\beta,\alpha^{'},\beta^{''}) + \mathcal{M}^{G}(\beta,\beta^{'},\alpha^{''}), \mathcal{M}^{G}(\beta,\beta^{'},\beta^{''}) \Big).$$ 
    So if $\mathcal{M}^{G} = 0,$ then $\mathcal{M}^{TG} = (0,0).$ We now assume that $\mathcal{M}^{TG} = (0,0)$ then for any $\beta,\beta^{'},\beta^{''} \in \mathfrak{g}^{*}$ wa have $$\mathcal{M}^{G}(\beta,\beta^{'})\beta^{''} = 0.$$
      Therefore $\mathcal{D}^{G}$ is metaflat.
      \item From theorem (\ref{th1}), the Levi-Civita contravariant connection $\mathcal{D}^{G}$ is falt if, and only if, 
      $\mathcal{D}^{TG}$ is falt. Moreover $\mathcal{D}^{G}$ is metafalt if, and only if, 
      $\mathcal{D}^{TG}$ is metafalt. Hence we deduce that $\{ , \}_{G}$ is a generalized Poisson bracket on $\Omega^{*}(M)$ if, and only if, 
       $\{ , \}_{TG}$ is a generalized Poisson bracket on $\Omega^{*}(TG)$.
        \end{enumerate}
      \end{proof}

     \section{Pseudo-Riemannian Sanchez de Alvarez tangent Poisson-Lie group}
     Note that in \cite{al-za} the author and N.Zaalani showed that the Sanchez de Alvarez tangent Poisson-Lie group $(TG,\Pi_{TG})$ of $G$ equipped with the natural left invariant Riemannian metric (or the left invariant Sasaki metric or the left invariant Cheeger-Gromoll metric) is a Riemannian Poisson-Lie group
     if, and only if, $(G,\Pi_{G})$ is a trivial Poisson-Lie group ($\Pi_{G} = 0$). Moreover, the semi-direct product Lie algebra $(\mathfrak{g} \rtimes \mathfrak{g}, [ , ]_{\mathfrak{g} \rtimes \mathfrak{g}})$ equipped with the natural scalar product is a Riemannian Lie algebra if, and only if, $(\mathfrak{g},[ , ]_{\mathfrak{g}})$ is an abelian Lie algebra.\\
     In this section we study the compatibility in the sense of M.Boucetta between the Sanchez de Alvarez Poisson-Lie structure $\Pi_{TG}$ structure and the complete pseudo-Riemannian metric $g^{c}$ on $TG$.\\
    
 Let $a$ be a bilinear, symmetric and non-degenerate form on $\mathfrak{g}^{*}$ and let $a^{c}$ be the complete lift of $a$ given for any $(\alpha,\beta),(\alpha^{'},\beta^{'}) \in \mathfrak{g}^{*} \times \mathfrak{g}^{*}$ by:
  \begin{equation*} 
\begin{array}{rcl}
 a^{c}((\alpha,0),(\alpha^{'},0))  &=& 0,\\
a^{c}((\alpha,0), (0,\beta^{'})) &=& (a(\alpha, \beta^{'}))^{v} \\
a^{c}((0,\beta), (0,\beta^{'})) &=& (a(\beta, \beta^{'}))^{c}.
\end{array}
\end{equation*}     
 Let $g$ be the left invariant contravariant pseudo-Riemannian metric associated to $a$ and let $g^{c}$ be the complete  left invariant contravariant pseudo-Riemannian metric associated to $a^{c}$.\\~~~~\\ 
 Let $(G,\Pi_{G})$ be a Poisson-Lie group with Lie bialgebra $(\mathfrak{g},\mathfrak{g}^{*})$. The linearized Poisson structure of $\Pi_{G}$ at $e$ is the linear Poisson structure $\Pi_{\mathfrak{g}}$ on $\mathfrak{g}$, whose value at $X \in \mathfrak{g}$ is given by: $$\Pi_{\mathfrak{g}}(X) = d_{e}\Pi_{G}(X).$$ The linear Poisson structure $\Pi_{\mathfrak{g}}$ on  $\mathfrak{g} = T_{e}G$, making $(\mathfrak{g},\Pi_{\mathfrak{g}})$ an abelian Poisson-Lie group with Lie bialgebra $(\mathfrak{g},\mathfrak{g}^{*})$ such that the  Lie bracket of $\mathfrak{g}$ is zero and the Lie bracket of $\mathfrak{g}^{*}$ is $[ , ]_{\mathfrak{g}^{*}}$. 
 \begin{rmk} \label{r1}
If $(G,\Pi_{G},g)$ is a pseudo-Riemannian Poisson Lie group then,  its dual Lie algebra $(\mathfrak{g}^{*},[ , ]_{\mathfrak{g}^{*}},a)$ equipped with the form $a$ is a pseudo-Riemannian Lie algebra and the abelian Poisson-Lie group $(\mathfrak{g},\Lambda_{\mathfrak{g}}, \tilde{a})$ equipped with the form $\tilde{a}$ associated to $a$ is a pseudo-Riemannian Poisson-Lie group \cite{Bouc}. 
\end{rmk}
 \begin{thm}
 Let $(G,\Pi_{G},g)$ be a Poisson-Lie group equipped with the left invariant contravariant pseudo-Riemannian metric and $(TG,\Pi_{TG},g^{c})$ the Sanchez de Alvarez tangent Poisson-Lie group of $G$ equipped with the complete left invariant pseudo-Riemannian  metric $g^{c}$.
 Then  $(G,\Pi_{G},g)$ is a pseudo-Riemannian Poisson-Lie group if, and only if,  $(TG,\Pi_{TG},g^{c})$ is a pseudo-Riemannian Poisson-Lie group. 
 \end{thm}
 \begin{proof}
 Note that the linear transformation $Ad^{*}_{g} : \mathfrak{g}^{*} \rightarrow \mathfrak{g}^{*}$ is a Lie algebra automorphism.\\ 
 The infinitesimal Levi-Civita connection $B$ associated to $([ , ]_{\mathfrak{g}^{*} \ltimes \mathfrak{g}^{*}}, a^{c})$ is given for any $(\alpha,\alpha^{'}), (\gamma,\gamma^{'}) \in \mathfrak{g}^{*} \times \mathfrak{g}^{*}$ by:
 $$B_{(\alpha,\alpha^{'})}(\gamma,\gamma^{'}) = (A_{\alpha}\gamma^{'} + A_{\alpha^{'}}\gamma,A_{\alpha^{'}}\gamma^{'}),$$
 where $A$ is the infinitesimal Levi-Civita connection associated respectively to $([ , ]_{\mathfrak{g}^{*}},a)$.\\
 For any $(X,Y) \in \mathfrak{g} \rtimes \mathfrak{g}$ and $(\gamma,\gamma^{'}) \in \mathfrak{g}^{*} \ltimes \mathfrak{g}^{*},$
 $$ad^{*}_{(X,Y)}(\gamma,\gamma^{'}) = (ad^{*}_{X}\gamma + ad^{*}_{Y}\gamma^{'},ad^{*}_{X}\gamma^{'}).$$
 Let $(x_{i})$ are local coordinates of $G$ in a neighborhood of $e$ and $(x_i,y_i)$ the correspondent local coordinates of  $TG$. The Poisson tensors of $G$ and $TG$ are expressed respectively by \cite{M.N}:
\begin{equation*}
\Pi_{G} =  \sum_{i,j} \Pi_{G}^{ij} \frac{\partial}{\partial_{x_i}} \wedge  \frac{\partial}{\partial_{x_j}},
\end{equation*}
and
\begin{equation} \label{poi}
\Pi_{TG} = \sum_{i,j,k} \Pi_{G}^{ij} \frac{\partial}{\partial_{x_i}} \wedge  \frac{\partial}{\partial_{y_j}} +   y_{k} \frac{\partial \Pi_{G}^{ij}}{\partial_{x_{k}}} \frac{\partial}{\partial_{y_i}} \wedge \frac{\partial}{\partial_{y_j}},
\end{equation}
then for any $(x,X) \in TG$ and for any $(\alpha,\alpha^{'}) \in \mathfrak{g}^{*} \times \mathfrak{g}^{*}$ we have,
 $$\Pi^{l}_{TG}(x,X)(\alpha,\alpha^{'}) = (\Pi^{l}_{G}(x)(\alpha^{'}),\Pi^{l}_{G}(x)(\alpha) + \Pi_{\mathfrak{g}}(X)(\alpha^{'})),$$
 where $\Pi_{\mathfrak{g}}$ is the linear Poisson structure on $\mathfrak{g}$ associated to $\Pi_{G}$.\\ 
 Then for any $(\alpha,\alpha^{'}),(\beta,\beta^{'}),(\gamma,\gamma^{'}) \in \mathfrak{g}^{*} \times \mathfrak{g}^{*}$ we obtain:
 $$\begin{array}{rcl}
 &&[B_{(\alpha,\alpha^{'})}(\gamma,\gamma^{'})+ ad^*_{\Pi^{l}_{TG}(x,X)(\alpha,\alpha^{'})}(\gamma,\gamma^{'}),(\beta,\beta^{'})]_{\mathfrak{g}^* \ltimes \mathfrak{g}^{*}}\\ &\quad \quad & + \quad [(\alpha,\alpha^{'}),B_{(\beta,\beta^{'})}(\gamma,\gamma^{'}) + ad^{*}_{\Pi^{l}_{TG}(x,X)(\beta,\beta^{'})}(\gamma,\gamma^{'})]_{\mathfrak{g}^* \ltimes \mathfrak{g}^{*}} \\ &=&
 [(A_{\alpha}\gamma^{'} + A_{\alpha^{'}} \gamma,A_{\alpha^{'}}\gamma^{'}) + ad^{*}_{(\Pi^{l}_{G}(x)(\alpha^{'}),\Pi^{l}_{G}(x)(\alpha) + \Pi_{\mathfrak{g}}(X)(\alpha^{'}))}(\gamma,\gamma^{'}),(\beta,\beta^{'})]_{\mathfrak{g}^{*} \ltimes \mathfrak{g}^{*}}\\ &\quad \quad & + \quad
 [(\alpha,\alpha^{'}),(A_{\beta}\gamma^{'} + A_{\beta^{'}} \gamma,A_{\beta^{'}}\gamma^{'}) + ad^{*}_{(\Pi^{l}_{G}(x)(\beta^{'}),\Pi^{l}_{G}(x)(\beta) + \Pi_{\mathfrak{g}}(X)(\beta^{'}))}(\gamma,\gamma^{'})]_{\mathfrak{g}^{*} \ltimes \mathfrak{g}^{*}} \\&=&
 [(A_{\alpha}\gamma^{'} + A_{\alpha^{'}} \gamma + ad^{*}_{\Pi^{l}_{G}(x)(\alpha^{'})}\gamma + ad^{*}_{\Pi^{l}_{G}(x)(\alpha)+ \Pi_{\mathfrak{g}}(X)(\alpha^{'})}\gamma^{'},A_{\alpha^{'}}\gamma^{'} + ad^{*}_{\Pi^{l}_{G}(x)(\alpha^{'})}\gamma^{'} ),(\beta,\beta^{'})]_{\mathfrak{g}^{*} \ltimes \mathfrak{g}^{*}}\\
 &\quad \quad & + \quad [(\alpha,\alpha^{'}),(A_{\beta}\gamma^{'} + A_{\beta^{'}} \gamma + ad^{*}_{\Pi^{l}_{G}(x)(\beta^{'})}\gamma + ad^{*}_{\Pi^{l}_{G}(x)(\beta)+ \Pi_{\mathfrak{g}}(X)(\beta^{'})}\gamma^{'}, A_{\beta^{'}}\gamma^{'} + ad^{*}_{\Pi^{l}_{G}(x)(\beta^{'})}\gamma^{'} )]_{\mathfrak{g}^{*} \ltimes \mathfrak{g}^{*}}\\&=& \Big([A_{\alpha}\gamma^{'} + A_{\alpha^{'}}\gamma +  ad^{*}_{\Pi^{l}_{G}(x)(\alpha^{'})}\gamma + ad^{*}_{\Pi^{l}_{G}(x)(\alpha)}\gamma^{'}+ ad^{*}_{\Pi_{\mathfrak{g}}(X)(\alpha^{'})}\gamma^{'},\beta^{'}]_{\mathfrak{g}^{*}}\\ &\quad \quad & + \quad [A_{\alpha^{'}}\gamma^{'} + ad^{*}_{\Pi^{l}_{G}(x)(\alpha^{'})}\gamma^{'},\beta]_{\mathfrak{g}^{*}},
 [A_{\alpha^{'}}\gamma^{'} + ad^{*}_{\Pi^{l}_{G}(x)(\alpha^{'})}\gamma^{'},\beta]_{\mathfrak{g}^{*}} \Big)
 + \Big([\alpha,A_{\beta^{'}}\gamma^{'} + ad^{*}_{\Pi^{l}_{G}(x)(\beta^{'})}\gamma^{'}]_{\mathfrak{g}^{*}}\\ &\quad \quad & + \quad [\alpha,A_{\beta}\gamma^{'} + A_{\beta^{'}}\gamma +  ad^{*}_{\Pi^{l}_{G}(x)(\beta^{'})}\gamma + ad^{*}_{\Pi^{l}_{G}(x)(\beta)}\gamma^{'}+ ad^{*}_{\Pi_{\mathfrak{g}}(X)(\beta^{'})}\gamma^{'}]_{\mathfrak{g}^{*}},[\alpha^{'},A_{\beta^{'}}\gamma^{'} + ad^{*}_{\Pi^{l}_{G}(x)(\beta^{'})}\gamma^{'} \Big) \\ &=& \Big([A_{\alpha}\gamma^{'} + ad^{*}_{\Pi^{l}_{G}(x)(\alpha)}\gamma^{'},\beta^{'}]_{\mathfrak{g}^{*}} + [\alpha,A_{\beta^{'}}\gamma^{'} + ad^{*}_{\Pi^{l}_{G}(x)(\beta^{'})}\gamma^{'}]_{\mathfrak{g}^{*}} + [A_{\alpha^{'}}\gamma + ad^{*}_{\Pi^{l}_{G}(x)(\alpha^{'})}\gamma,\beta^{'}]_{\mathfrak{g}^{*}}\\ &\quad \quad & + \quad [\alpha^{'},A_{\beta^{'}}\gamma + ad^{*}_{\Pi^{l}_{G}(x)(\beta^{'})}\gamma]_{\mathfrak{g}^{*}}  + [A_{\alpha^{'}}\gamma^{'} + ad^{*}_{\Pi^{l}_{G}(x)(\alpha^{'})}\gamma^{'},\beta]_{\mathfrak{g}^{*}} + [\alpha^{'},A_{\beta}\gamma^{'} + ad^{*}_{\Pi^{l}_{G}(x)(\beta)}\gamma^{'}]_{\mathfrak{g}^{*}}\\ &\quad \quad & + \quad [ ad^{*}_{\Pi_{\mathfrak{g}}(X)(\alpha^{'})}\gamma^{'},\beta^{'}]_{\mathfrak{g}^{*}} + [\alpha^{'}, ad^{*}_{\Pi_{\mathfrak{g}}(X)(\beta^{'})}\gamma^{'}]_{\mathfrak{g}^{*}}, [A_{\alpha^{'}}\gamma^{'} + ad^{*}_{\Pi^{l}_{G}(x)(\alpha^{'})}\gamma^{'},\beta^{'}]_{\mathfrak{g}^{*}}\\ &\quad \quad & + \quad [\alpha^{'},A_{\beta^{'}}\gamma^{'} + ad^{*}_{\Pi^{l}_{G}(x)(\beta^{'})}\gamma^{'}]_{\mathfrak{g}^{*}}\Big).
 \end{array}$$
 Then from remark (\ref{r1}), if $(G,\Pi_{G},g)$ is a pseudo-Riemannian Poisson-Lie group then  $(TG,\Pi_{TG},g^{c})$ is a pseudo-Riemannian Poisson-Lie group.  Conversely, if  $(TG,\Pi_{TG},g^{c})$ is a pseudo-Riemannian Poisson-Lie group then for any $x \in G$ and for any $\alpha^{'}, \beta^{'}, \gamma^{'} \in \mathfrak{g}^{*}$ we have 
 $$ [A_{\alpha^{'}}\gamma^{'} + ad^{*}_{\Pi^{l}_{G}(x)(\alpha^{'})}\gamma^{'},\beta^{'}]_{\mathfrak{g}^{*}} +  [\alpha^{'},A_{\beta^{'}}\gamma^{'} + ad^{*}_{\Pi^{l}_{G}(x)(\beta^{'})}\gamma^{'}]_{\mathfrak{g}^{*}} = 0,$$
 therefore $(G,\Pi_{G},g)$ is a pseudo-Riemannian Poisson-Lie group.
 \end{proof}
 \begin{prop}
Let  $(\mathfrak{g},[ , ]_{\mathfrak{g}},\tilde{a})$ be a Lie algebra equipped with  a bilinear, symmetric and non-degenerate form $\tilde{a}$ and let $(\mathfrak{g} \rtimes \mathfrak{g}, [ , ]_{\mathfrak{g} \rtimes \mathfrak{g}},\tilde{a}^{c})$ be  the semi-direct product Lie algebra equipped with the complete form $\tilde{a}^{c}.$ Then $(\mathfrak{g},[ , ]_{\mathfrak{g}},\tilde{a})$ is a pseudo-Riemannian Lie algebra if, and only if,  $(\mathfrak{g} \rtimes \mathfrak{g}, [ , ]_{\mathfrak{g} \rtimes \mathfrak{g}},\tilde{a}^{c})$ is a  pseudo-Riemannian Lie algebra.
\end{prop}
\begin{proof}
The infinitesimal Levi-Civita connection $B$ associated to $([ , ]_{\mathfrak{g} \rtimes \mathfrak{g}},\tilde{a}^{c})$ is given for any $(X,Y),(X^{'},Y^{'}) \in \mathfrak{g} \rtimes \mathfrak{g},$
\begin{equation*}
B_{(X,Y)}(X^{'},Y^{'}) = (A_{X}X^{'},A_{X}Y^{'} + A_{Y}X^{'}),
\end{equation*}
where $A$ is the infinitesimal Levi-Civita associated to $([ , ]_{\mathfrak{g}},\tilde{a}).$\\
According to equation (\ref{Riemannian Lie}), for any  $(X,Y),(X^{'},Y^{'}), (X^{''},Y^{''}) \in \mathfrak{g} \rtimes \mathfrak{g},$ we obtain :
$$\begin{array}{rcl}
&&[B_{(X,Y)}(X^{'},Y^{'}), (X^{''},Y^{''})]_{\mathfrak{g} \rtimes \mathfrak{g}} + [(X,Y),B_{(X^{''},Y^{''})} (X^{'},Y^{'})]_{\mathfrak{g} \rtimes \mathfrak{g}}\\ &=& \Big([A_{X}X^{'},X^{''}]_{\mathfrak{g}} + [X,A_{X^{''}}X^{'}]_{\mathfrak{g}}, [A_{X}X^{'},Y^{''}]_{\mathfrak{g}} + [X,A_{Y^{''}}X^{'}]_{\mathfrak{g}} + [A_{X}Y^{'},X^{''}]_{\mathfrak{g}} + [X,A_{X^{''}}Y^{'}]_{\mathfrak{g}} \\ &\quad \quad & + \quad  [A_{Y}X^{'},X^{''}]_{\mathfrak{g}} + [Y,A_{X^{''}}X^{'}]_{\mathfrak{g}} \Big).
\end{array}$$
\end{proof}
\section{Examples}
\begin{enumerate}
 \item
Let $(e_{1},e_{2},e_{3})$ be an orthonormal basis of $\mathbb{R}^{3}$. The Lie algebra $\mathbb{R}^{3}$ with bracket,
$$[e_{1},e_{2}]_{\mathbb{R}^{3}} = \lambda e_{3}, \hspace{2mm} [e_{1},e_{3}]_{\mathbb{R}^{3}} = -\lambda e_{2}, \hspace{2mm} [e_{2},e_{3}]_{\mathbb{R}^{3}} = 0, \hspace{2mm} \lambda< 0,$$
is a Riemannian Lie algebra \cite{Bouc}. The infinitesimal situation can be integrated and we obtain that the 
 triplet $(\mathbb{R}^{3}, \Pi_{\mathbb{R}^{3}}, \langle , \rangle_{\mathbb{R}^{3}})$ is a Riemannian Poisson Lie group, where $\mathbb{R}^{3}$ is equipped with its abelian Lie group structure, $\langle , \rangle_{\mathbb{R}^{3}}$ its canonical Euclidian metric and\\
$$\Pi_{\mathbb{R}^{3}} = \lambda \frac {\partial } {\partial_{x}} \wedge
(z \frac {\partial } {\partial_{y}} - y \frac {\partial } {\partial_{z}}), \hspace{2mm}\lambda < 0.$$

Now let $(x_{i})$ are local coordinates of $G$ in a neighborhood of $e$ and $(x_i,y_i)$ the correspondent local coordinates of  $TG$.
The Riemannian metrics on $G$ and $TG$ is expressed respectively by \cite{ML-PR}:

$$ g = \sum_{i,j} g^{ij} dx_{i} \otimes dx_{j}$$

and 
\begin{equation} \label{met}
g^{c} =  \sum_{i,j,k} y_{k} \frac{\partial g^{ij}}{\partial_{x^{k}}} dx_{i} \otimes dx_{j} + g^{ij} dx_{i} \otimes dy_{j} + g^{ij} dy_{i} \otimes dx_{j}
\end{equation}
 Then using equations (\ref{poi}) and (\ref{met}), the 6-dimensional Poisson-Lie group $(T\mathbb{R}^{3} \equiv \mathbb{R}^{6},\Pi_{\mathbb{R}^{6}}, \langle , \rangle_{\mathbb{R}^{6}})$, where $\mathbb{R}^{6}$ is equipped with its abelian Lie group structure with coordinate $(x, y, z, u, v,w)$,
 $$\Pi_{\mathbb{R}^{6}} = \lambda \frac {\partial } {\partial_{x}} \wedge
(z \frac {\partial } {\partial_{v}} - y \frac {\partial } {\partial_{w}}) +  \lambda \frac {\partial } {\partial_{u}} \wedge
(w \frac {\partial } {\partial_{v}} - v \frac {\partial } {\partial_{w}})$$
and
$$\langle , \rangle_{\mathbb{R}^{6}} = dxdu + dydv + dzdw + dudx + dvdy + dwdz$$
is a pseudo-Riemannian Poisson-Lie group.
\item The 3-dimensional Heisenberg Lie algebra $$  \mathfrak{h}_{3} = \left \lbrace \begin{array}{rcl} \left  ( \begin{array}{rcl} 0 & x & z  \\ 0 & 0 & y \\ 0 & 0 & 0  \end{array}
   \right ),(x,y,z) \in \mathbb{R}^{3}
   \end{array}
   \right\rbrace,$$
   with bracket $[e_{1},e_{2}]_{\mathfrak{h}} = e_{3}$ and $[e_{1},e_{3}]_{\mathfrak{h}} = [e_{2},e_{3}]_{\mathfrak{h}} = 0$ is a pseudo-Riemannian Lie algebra \cite{Bouc}. Then the semi-direct product Lie algebra $\mathfrak{h}_{3} \rtimes \mathfrak{h}_{3}$ of dimension $6$ given by,
   $$ \mathfrak{h}_{3} \rtimes \mathfrak{h}_{3} = \left \lbrace \Big (\begin{array}{rcl}  \left  ( \begin{array}{rcl} 0 & x & z  \\ 0 & 0 & y \\ 0 & 0 & 0  \end{array}
   \right ) , \left  ( \begin{array}{rcl} 0 & u & w  \\ 0 & 0 & v \\ 0 & 0 & 0  \end{array}
   \right )\Big),(x,y,z,u,v,w) \in \mathbb{R}^{6}
   \end{array}
   \right\rbrace,$$
   with bracket 
   $$[(e_{i},e^{'}_{j}),(e_{k},e^{'}_{l})]_{\mathfrak{h}_{3} \rtimes \mathfrak{h}_{3}} = ([e_{i},e_{k}]_{\mathfrak{h}},[e_{i},e^{'}_{l}]_{\mathfrak{h}} + [e^{'}_{j},e_{k}]_{\mathfrak{h}}), \hspace{3mm}  1\leqslant i,j,k,l \leqslant 3.$$
   is a pseudo-Riemannian Lie algebra, where  $$[e_{1},e^{'}_{2}]_{\mathfrak{h}} = e^{'}_{3}, [e_{2},e^{'}_{1}]_{\mathfrak{h}} = -e^{'}_{3}, [e_{1},e^{'}_{3}]_{\mathfrak{h}} = [e_{2},e^{'}_{3}]_{\mathfrak{h}} = [e_{3},e^{'}_{1}]_{\mathfrak{h}} = [e_{3},e^{'}_{1}]_{\mathfrak{h}} = 0.$$
\end{enumerate}
\textbf{Acknowledgement}. I would like to thank Eli Hawkins and Nadhem Zaalani for
their fruitful and helpful discussions.

\end{document}